\documentclass[12pt]{amsart}
\usepackage{graphicx} 
\usepackage[utf8]{inputenc}
\usepackage{amsfonts}
\usepackage{amssymb}
\usepackage{amsmath}
\usepackage{bm}
\usepackage{amsthm}
\usepackage[T1]{fontenc}
\usepackage[english]{babel}
\usepackage{hyperref}
\usepackage{geometry}
\usepackage{bbold}
\usepackage{tikz}
\usepackage{pgf,tikz}
\usetikzlibrary{shapes}
\usetikzlibrary{arrows.meta}
\usetikzlibrary{graphs} 
\usetikzlibrary{graphs,quotes} 
\usetikzlibrary{graphs,shapes.geometric}
\usetikzlibrary{decorations.markings}
\tikzstyle directed=[postaction={decorate,decoration={markings,
    mark=at position .65 with {\arrow{latex}}}}]
\usepackage{stmaryrd}
\usepackage[all]{xy}
\newtheorem{theorem}{Theorem}[section]
\newtheorem{corollary}[theorem]{Corollary}
\newtheorem{lemma}[theorem]{Lemma}
\newtheorem{proposition}[theorem]{Proposition}
\theoremstyle{definition}

\newtheorem{remark}[theorem]{Remark}
\newtheorem{example}[theorem]{Example}

\newtheorem{notation}[theorem]{Notation}
\newcommand\Sub{\mathrm{Sub}}
\newcommand\Stab{\mathrm{Stab}}
\newcommand\Ph{\mathrm{Ph}}
\newcommand\dom{\mathrm{dom}}
\newcommand\rng{\mathrm{rng}}
\newcommand\trg{\mathbf{t}}
\newcommand\src{\mathbf{s}}

\newcommand\BS{\mathrm{BS}}
\newcommand\Sch{\mathrm{Sch}}

\newcommand\id{\mathrm{id}}
\newcommand\Supp{\mathrm{Supp}}
\newcommand\Cay{\mathrm{Cay}}
\usepackage{bm}
\title{Subgroup mixing in Baumslag-Solitar groups}
\author{Sasha Bontemps}
\date{\today}

\newenvironment{acknowledgements}{%
	\begin{abstract}
	}{%
	\end{abstract}
}

\begin{document}
	
	\maketitle
	
	\begin{abstract}
	In this article, we contribute to the study of the dynamics by conjugation on the space of subgroups of Baumslag-Solitar groups $\BS(m,n)$, via the mixing properties of elements asymptotically produced by suitable random walks on the group.
In an acylindrically hyperbolic context, the authors of \cite{osin} demonstrated strong mixing situations, namely topological $\mu$-mixing, a strengthening of high topological transitivity.
Regarding non-metabelian $\BS(m,n)$ with $|m|\neq |n|$, we exhibit here a radically different situation on each of the pieces except one of the partition introduced in \cite{solitar} (although it is highly topologically transitive on each piece). On the other hand, when $|m|= |n|$, we demonstrate the topological $\mu$-mixing character of the action on each of the pieces. 
	\end{abstract}
	
	\section{Introduction}

	Given a couple $(m,n) \in \mathbb{Z}^*$, the Baumslag-Solitar group of parameter $(m,n)$ is the group defined by the following presentation \begin{equation}\label{baumsol}\BS(m,n) = \langle b, t \mid tb^mt^{-1} = b^n \rangle.\end{equation}
	Baumslag-Solitar groups were introduced in \cite{tworelators} to give the first examples of two-generated, finitely presented non-Hopfian groups. They have been widely studied in relation to various properties, that strongly depend on the parameters $(m,n)$: residual finiteness (\cite{meskin}), classification up to quasi-isometry (see \cite{farb} and \cite{whyte}), classification up to measure equivalence (announced by the authors of \cite{measureeq})...The group $\BS(m,n)$ acts on its \textbf{Bass-Serre tree} $\mathcal{T}_{m,n}$ (\textit{i.e.} the infinite oriented tree all of whose vertices have $m$ incoming edges and $n$ outgoing edges), with a single orbit of vertices and a single orbit of edges, and the vertex and edge stabilizers are infinite cyclic.  
	
	In this article, we pursue the study of the space of subgroups of Baumslag-Solitar groups, which was initiated in \cite{solitar}. Endowed with the Chabauty topolology, the set of subgroups $\Sub(\Gamma)$ of any infinite countable group $\Gamma$ is a closed subset of the Cantor space $\{0,1\}^{\Gamma}$. A particular subset of $\Sub(\Gamma)$ that we are interested in is the \textbf{perfect kernel} $\mathcal{K}(\Gamma)$ of $\Gamma$, \textit{i.e.} the largest closed subset without isolated point. This subset is invariant under $\Gamma$-conjugation. We are interested in the dynamics induced by this action. More precisely, we are interested in finding subsets of the perfect kernel on which the action is highly topologically transitive, or even topologically $\mu$-mixing. Recall that an action of a group $\Gamma$ on a topological space $X$ is \textbf{highly topologically transitive} (HTT) if for every $r \in \mathbb{N}$ and for every non-empty open subsets $U_1,...,U_r$, $V_1,...,V_r$, there exists $g \in \Gamma$ such that $g \cdot U_i \cap V_i \neq \emptyset$ for every $i \in \llbracket 1,r \rrbracket$. The authors of \cite{osin} introduced a strenghtening of this notion, called topological $\mu$-mixing. Given a probability measure $\mu$ on a countable group $\Gamma$, an action of $\Gamma$ on a topological space $X$ is called \textbf{topologically $\mu$-mixing} if for every non-empty open subsets $U,V \subseteq X$, denoting by $(S_k)_{k \in \mathbb{N}}$ a sequence of independently $\mu$-distributed random variables and by $(G_k)_{k \in \mathbb{N}}=(S_1 \cdots S_k)_{k \in \mathbb{N}}$ the random walk on $\Gamma$ with step distribution $\mu$, one has \[\lim_{k \to \infty}\mathbb{P}\left(G_k \cdot U \cap V \neq \emptyset\right) = 1.\] 
	
	The authors of \cite{AG} and \cite{osin} studied independently the space $\Sub(\Gamma)$ in the case where $\Gamma$ acts on a hyperbolic space with "vanishing" stabilizers. More formally, in the case where the hyperbolic space is a tree $\mathcal{T}$ (and the action is minimal and irreducible), one application of their results we are interested in is the case where the action of $\Gamma$ on $\mathcal{T}$ is \textbf{acylindrical}, \textit{i.e.} there exists $R > 0$ such that the stabilizer of any path of length larger than $R$ is trivial. The following statements are applications of their results to this particular setting. In \cite{AG}, the authors proved that in this case, the perfect kernel of $\Gamma$ contains the closure of the set $\Sub_{| \bullet \backslash \mathcal{T}|_{\infty}}(\Gamma)$ of subgroups $\Lambda$ of $\Gamma$ satisfying that the quotient graph $\Lambda \backslash \mathcal{T}$ is infinite: \[\overline{\Sub_{| \bullet \backslash \mathcal{T}|_{\infty}}(\Gamma)} \subseteq \mathcal{K}(\Gamma).\] Moreover, this subset is invariant under conjugation and the action of $\Gamma$ on $\overline{\Sub_{| \bullet \backslash \mathcal{T}|_{\infty}}(\Gamma)}$ is HTT. Still in the particular context of acylindrical actions on trees, the results of \cite{osin} lead to the study of the set $\Sub_{\infty}^{cc}(\Gamma \curvearrowright \mathcal{T})$, namely the set of infinite index $\mathcal{T}$-convex cocompact subgroups. Recall that a subgroup $\Lambda \leq \Gamma$ is called \textbf{$\mathcal{T}$-convex cocompact} if it acts properly on $\mathcal{T}$ (\textit{i.e.} with finite vertex stabilizers), with quasi-convex orbits (\textit{i.e.} for any vertex $v \in \mathcal{T}$, there exists $\eta > 0$ such that the reduced edge path connecting two vertices of $\Lambda \cdot v$ remains at distance $< \eta$ from $\Lambda \cdot v$). In the setting of actions on simplicial trees, this subset is exactly the set of finitely generated subgroups of infinite index that act properly on $\mathcal{T}$. If we assume moreover that the action is acylindrical, this is a subset of $\Sub_{|\bullet \backslash \mathcal{T}|_{\infty}}(\Gamma)$: \[\Sub_{\infty}^{cc}(\Gamma \curvearrowright \mathcal{T}) \subseteq \Sub_{|\bullet \backslash \mathcal{T}|_{\infty}}(\Gamma). \] On the closure of this subset, they proved that the action by conjugation is even topologically $\mu$-mixing for every measure $\mu$ on $\Gamma$ whose support is bounded (\textit{i.e.} $\{g \cdot v, g \in \Supp(\mu)\}$ is finite for any vertex $v$ of $\mathcal{T}$), symmetric (\textit{i.e.} stable under inversion), and generates $\Gamma$.
	
	A Baumslag-Solitar group $\BS(m,n)$ is a typical example whose action on its Bass-Serre tree $\mathcal{T}_{m,n}$ is \textit{not} acylindrical, because the stabilizer of every finite edge path is infinite cyclic. The authors of \cite{solitar} and \cite{solitar2} proved that this leads to a very different situation for the dynamics induced by the action by conjugation on the perfect kernel. In the case where $\min(|m|,|n|) > 1$, they proved that the perfect kernel exactly consists of the set $\Sub_{| \bullet \backslash \mathcal{T}_{m,n}|_{\infty}}(\BS(m,n))$, and they uncovered a countably infinite invariant partition of the perfect kernel $\mathcal{K}(\BS(m,n)) = \bigsqcup_{P \in \mathcal{Q}_{m,n}}\mathcal{K}_P$ (where $\mathcal{Q}_{m,n}$ is an infinite subset of $\mathbb{N}^* \sqcup \{\infty\}$ that contains $\infty$) such that \begin{itemize}
		\item $\mathcal{K}_P$ is open for every finite $P \in \mathcal{Q}_{m,n}$ (and also closed iff $|m|=|n|$);
		\item $\mathcal{K}_{\infty}$ is closed;
		\item the action by conjugation on $\mathcal{K}_P$ is HTT for every $P \in \mathcal{Q}_{m,n}$.
	\end{itemize}
	Notice that the existence of disjoint invariant open subsets prevents the action on $\mathcal{K}(\BS(m,n))$ from being HTT. However, the last item may make us wonder if the action is also topologically $\mu$-mixing on each piece. Our first result shows that this is false in general: \begin{theorem}\label{nonmixing}
		Let $m,n \in \mathbb{Z}$ such that $\min(|m|,|n|) > 1$. Let us assume that $|m| \neq |n|$. Then, there exists a probability measure $\mu$ whose support is finite, symmetric and generates $\BS(m,n)$, such that for every finite $P \in \mathcal{Q}_{m,n}$, the action by conjugation of $\BS(m,n)$ on $\mathcal{K}_P$ is not topologically $\mu$-mixing.
	\end{theorem}
	Notice that, in the proof of Theorem \ref{nonmixing}, though the support of $\mu$ is symmetric, we will construct $\mu$ in such a way that $\mu(t) \neq \mu(t^{-1})$. 
	
	However, we have the following positive result: \begin{theorem}\label{mixing}
		Let $m,n \in \mathbb{Z}$ such that $\min(|m|, |n|) > 1$ and let $\mathcal{T}_{m,n}$ be the Bass-Serre tree of $\BS(m,n)$. Let $\mu$ be a probability measure on $\BS(m,n)$ whose support is bounded, symmetric, and generates $\BS(m,n)$. Then: \begin{enumerate}
			\item the action by conjugation of $\BS(m,n)$ on $\mathcal{K}_{\infty}$ is topologically $\mu$-mixing;
			\item if $|m|=|n|$, then for every $P \in \mathcal{Q}_{m,n}$, the action by conjugation of $\BS(m,n)$ on $\mathcal{K}_P$ is topologically $\mu$-mixing. 
		\end{enumerate}
	\end{theorem}
	Notice that, for every $(m,n) \in \mathbb{Z}^2$ such that $\min(|m|,|n|) > 1$, the set $\mathcal{K}_{\infty}$ exactly consists of $\overline{\Sub_{\infty}^{cc}(\BS(m,n) \curvearrowright \mathcal{T}_{m,n})}$. Thus we extend the aforementioned result of \cite{osin} to this particular case of a non-acylindrical action. 
	
	Given a group $\Gamma$ acting on a tree $\mathcal{T}$, the following array summarizes the results we mentioned: \\
	\\
	\begin{tabular}{c|c|c|c}
		& {\small $\Gamma \curvearrowright \mathcal{T}$ acylindrical} & \multicolumn{2}{l}{{\small $\BS(m,n) \curvearrowright \mathcal{T}_{m,n}$ ($|m|, |n| \geq 2$) not acylindrical}} \\
		&   & {\small Negative results} & {\small Positive results} \\
		\hline
		& & & \\
		{\small HTT} & {\small $\Gamma \curvearrowright \overline{\Sub_{| \bullet \backslash \mathcal{T}|_{\infty}}(\Gamma)}$ HTT} &  {\small $\BS(m,n) \curvearrowright \mathcal{K}(\BS(m,n))$} & {\small $\BS(m,n) \curvearrowright \mathcal{K}_P$ HTT} \\
		& {\small \cite{AG}} & {\small $= \Sub_{| \bullet \backslash \mathcal{T}_{m,n}|_{\infty}}(\BS(m,n))$} & {\small $ \forall P \in \mathcal{Q}_{m,n}$.} \\
		&  & {\small not HTT} & {\small \cite{solitar2}} \\
		& & {\small \cite{solitar}} & \\
		\hline 
		& & & \\
		{\small $\mu$-mixing} & {\small $\Gamma \curvearrowright \overline{\Sub^{cc}_{[\infty]}(\Gamma \curvearrowright \mathcal{T})}$} & {\small $\BS(m,n) \curvearrowright \mathcal{K}_P$} & {\small $\BS(m,\pm m) \curvearrowright \mathcal{K}_P$ $\mu$-mixing} \\
		& {\small $\subseteq \overline{\Sub_{| \bullet \backslash \mathcal{T}|_{\infty}}(\Gamma)}$ $\mu$-mixing.} & {\small not $\mu$-mixing} & {\small $\forall P \in \mathcal{Q}_{m,n}\cap \mathbb{N}^*$;} \\
		& {\small \cite{osin}} & {\small in general if $|m| \neq |n|$.} & {\small $\BS(m,n) \curvearrowright \mathcal{K}_{\infty}$} \\
		&  & {\small Theorem \ref{nonmixing}} & {\small $ = \overline{\Sub^{cc}_{\infty}(\BS(m,n) \curvearrowright \mathcal{T}_{m,n})}$} \\
		& & & {\small also $\mu$-mixing $\forall m,n$.} \\
		&  &  & {\small Theorem \ref{mixing}}
		
	\end{tabular} 
 \\
	\\
	
	The paper is organized as follows. First we recall some background around Baumslag-Solitar groups and Bass-Serre theory. We recall the main tools and the decomposition of the perfect kernel of $\BS(m,n)$ introduced in \cite{solitar}. Then, we build a measure $\mu$ supported on $\{b,b^{-1},t,t^{-1}\}$ such that $\mu(t) \neq \mu(t^{-1})$ to prove Theorem \ref{nonmixing}, and we finally prove Theorem \ref{mixing} using a result of Cartwright and Soardi.
	
	\begin{acknowledgements}
		I express my gratitude to my PhD advisor Damien Gaboriau for his supervision and for valuable suggestions during the writing of this work. I also warmly thank Grégory Miermont and Mathieu Mourichoux for their enlightening explanations about random walks. I thank Martín Gilabert Vio and Mathieu Mourichoux for insightful discussions that helped me to correct a mistake in a previous version of that paper. I also thank Rémi Coulon and Ashot Minasyan for their careful reading that significantly improved the quality of this text. Finally, I warmly thank the anonymous referee for valuable suggestions that significantly improved the clarity of this paper. This work was supported by a CDSN from ENS de Lyon.
	\end{acknowledgements}
	
	\tableofcontents
	
	\section{Preliminaries and notations}
	
	We denote by $\mathcal{P}$ the set of prime numbers. For every integer $N$ and every prime number $p$, we denote by $|N|_p$ the $p$-adic valuation of $N$, that is, the largest $n \in \mathbb{N}$ such that $p^n$ divides $N$. For every integers $m,n \in \mathbb{Z} \smallsetminus \{0\}$ we denote by $m \wedge n$ their greatest common divisor, that is, the largest integer $k \in \mathbb{N}$ dividing both $m$ and $n$. By convention, $\infty \wedge n = n \wedge \infty = |n|$ for every $n \in \mathbb{N}^*$. By "countable" we mean finite or in bijection with $\mathbb{N}$. For any finite set $F$, we denote by $|F|$ its cardinality. We use the same convention as in \cite{serre} about graphs. We denote by $\mathcal{V}(\mathcal{G})$ the set of vertices of an oriented graph $\mathcal{G}$ and by $\mathcal{E}(\mathcal{G})$ its set of edges. We denote by $\mathcal{E}^+(\mathcal{G})$ (\textit{resp.} $\mathcal{E}^-(\mathcal{G})$) its set of positive (\textit{resp.} negative) edges, and by $\src : \mathcal{E}(\mathcal{G}) \to \mathcal{V}(\mathcal{G})$ and $\trg : \mathcal{E}(\mathcal{G}) \to \mathcal{V}(\mathcal{G})$ the source and target maps, respectively. For any edge $e$, we denote by $\overline{e}$ the reversed edge. If $\mathcal{G}$ is a graph, we denote by $d_{\mathcal{G}}$ the induced metric on the set of vertices. If $v \in \mathcal{V}(\mathcal{G})$ and $R \in \mathbb{N}^*$, we denote by $B_{\mathcal{G}}(v,R)$ the closed ball of center $v$ and radius $R$ in $\mathcal{G}$ (for the metric $d_{\mathcal{G}}$). For any subgraph $\mathcal{G}_0$ of $\mathcal{G}$ and any $R \in \mathbb{N}^*$, we denote by $\mathcal{N}_{\mathcal{G},R}(\mathcal{G}_0)$ the $R$-neighborhood of $\mathcal{G}_0$ in $\mathcal{G}$, \textit{i.e.} the subgraph of $\mathcal{G}$ induced by the set of vertices at distance less than $R$ from $\mathcal{G}_0$.  By the half-graph of an edge $e$, we mean the connected component of $\mathcal{G} \smallsetminus \{e\}$ that contains $\trg(e)$. If this half-graph is a tree, we call it a half-tree. If $\mathcal{T}$ is a tree and $\mathcal{T}^{(0)}$ is a subtree of $\mathcal{T}$ and $e \in \mathcal{E}\left(\mathcal{T} \smallsetminus \mathcal{T}^{(0)}\right)$, we say that $e$ points towards $\mathcal{T}^{(0)}$ if the half-tree of $e$ contains $\mathcal{T}^{(0)}$.
	
	\subsection{Space of subgroups}
	
	Given an infinite countable group $\Gamma$, one can endow its set of subgroups $\Sub(\Gamma)$ with the \textbf{Chabauty topology}, which comes from the natural inclusion $\Sub(\Gamma) \hookrightarrow \{0,1\}^{\Gamma}$ of $\Sub(\Gamma)$ into the Cantor space. A basis of neighborhoods is given by the following family of clopen sets \[\mathcal{V}(O,I) = \{\Lambda \leq \Gamma \mid I \subseteq \Lambda \text{ \ and \ } \Lambda \cap O = \emptyset \}\]
	(where $I$ and $O$ are finite subsets of $\Gamma$). 
	
	One has a correspondence between subgroups of $\Gamma$ and isomorphism classes of transitive right actions of $\Gamma$ on pointed countable sets, which yields a correspondence between conjugacy classes of subgroups of $\Gamma$ and isomorphism classes of transitive right actions of $\Gamma$ on countable sets, where the action by conjugation amounts to changing the base point. It is given by the bijection \begin{equation}\label{corr}\begin{array}{ccccc}
		\{\text{isomorphism classes of pointed transitive right $\Gamma$-actions}\} & \to & \Sub(\Gamma) \\
		(X,x_0) \curvearrowleft^{\alpha} \Gamma & \mapsto & \Stab_{\alpha}(x_0) \\
	\end{array}\end{equation}
	whose inverse is given by \[\begin{array}{ccccc}
		\Sub(\Gamma) & \to & \{\text{isomorphism classes of pointed transitive right $\Gamma$-actions}\} \\
		\Lambda & \mapsto & (\Lambda \backslash \Gamma, \Lambda) \curvearrowleft \Gamma \\
	\end{array}.\]
	
	The Chabauty topology can be defined on the set of isomorphism classes of pointed transitive right $\Gamma$-actions thanks to \textbf{Schreier graphs}. Given a symmetric generating set $S$ of $\Gamma$ and a right $\Gamma$-action $\alpha$ on a pointed countable set $(X,x_0)$, one can define the (rooted) Schreier graph of $\alpha$ as follows: its set of vertices is $X$ and, for every $x \in X$ and $s \in S$, there is a positive edge labeled $s$ with source $x$ and target $x \cdot s$, whose opposite edge is labeled $s^{-1}$ (and has source $x \cdot s$ and target $x$). Its root is the point $x_0$. The set of isomorphism classes of pointed transitive right $\Gamma$-actions can be endowed with the following topology. A basis of neighborhoods of a pointed transitive right action  $(X,v) \curvearrowleft^{\alpha} \Gamma$ is given by the set of actions whose Schreier graph has the same $R$-ball around the origin (for $R > 0$): \begin{align*}V_R = &\{(X',v') \curvearrowleft^{\beta} \Gamma, (B_{\Sch(\beta)}(v',R),v') \text{ and } (B_{\Sch(\alpha)}(v, R), v) \\ &\text{ are isomorphic (as pointed labeled graphs)}.\}\end{align*}
	When $S$ is finite, this is exactly the Chabauty topology (\textit{via} the aforementioned identification between actions and subgroups). We refer to \cite[Section 2]{bontemps} for more details. From now on, we will freely identify the set of subgroups of $\Gamma$ with the set of isomorphism classes of pointed transitive right actions of $\Gamma$. 
	
	By a theorem of Cantor-Bendixson, there exists a unique decomposition \[\Sub(\Gamma) = \mathcal{K}(\Gamma) \sqcup C\]
	where $C$ is countable and $\mathcal{K}(\Gamma)$ is a closed subset without isolated point, called the \textbf{perfect kernel} of $\Gamma$. This is the largest closed subset of $\Sub(\Gamma)$ without isolated points, or equivalently, the set of subgroups all of whose neighborhoods are uncountable. See \cite[Section 6]{Kechris} for more details.
	
	\subsection{Baumslag-Solitar groups; preactions and $(m,n)$-(Schreier) graphs}
	
	In this section, we provide a quick reminder on Bass-Serre theory and we recall the main tools that were introduced in \cite{solitar} to study $\Sub(\BS(m,n))$. Let $m,n \in \mathbb{Z}$ such that $\min(|m|,|n|) > 1$. The Baumslag-Solitar group of parameters $(m,n)$ is the group $\BS(m,n)$ defined by the presentation \eqref{baumsol}.
	
	\subsubsection{Some background on Bass-Serre theory}
	
	As an HNN-extension, the group $\BS(m,n)$ acts on its \textbf{Bass-Serre tree} $\mathcal{T}_{m,n}$ (on the left): this is the infinite oriented tree all of whose vertices have $m$ incoming edges and $n$ outgoing edges. This tree arises as follows: it is obtained from the (right) Cayley graph of $\BS(m,n)$ (with respect to the generating set $\{b,b^{-1},t,t^{-1}\}$) by shrinking all the $\langle b \rangle$-orbits. There is a natural map $p : \Cay(\BS(m,n)) \to \mathcal{T}_{m,n}$ (applying \cite[Definition 3.10]{solitar} to the free action of $\BS(m,n)$ on itself); it sends \begin{itemize}
		\item the vertex $\gamma$ of the Cayley graph to the vertex $\gamma \langle b \rangle$ of the Bass-Serre tree;
		\item the positive edge $(\gamma, \gamma  t)$ of the Cayley graph to the positive edge $\gamma \langle b^n \rangle$ of the Bass-Serre tree;
		\item the negative edge $(\gamma, \gamma  t^{-1})$ of the Cayley graph to the negative edge $\gamma \langle b^m \rangle$ of the Bass-Serre tree;
		\item the edges of the form $(\gamma, \gamma b^k)$ to the single vertex $\gamma \langle b \rangle$.
	\end{itemize}
	The action of $\BS(m,n)$ on the set of (positive) edges of $\mathcal{T}_{m,n}$ is defined by: $\gamma \cdot g\langle b^n\rangle = \gamma g \langle b^n \rangle$. The action of $\BS(m,n)$ on its Bass-Serre tree has a single orbit of vertices and the quotient graph $\BS(m,n) \backslash \mathcal{T}_{m,n}$ is a loop. The vertex $p(1) = \langle b \rangle$ is the root of the Bass-Serre tree. See \cite[Section 2.3, Section 3]{solitar} for more details.
	
	\subsubsection{Normal forms}\label{normalform}
	
	A \textbf{word} of $\BS(m,n)$ is a finite sequence of elements of $\{b,b^{-1}, t, t^{-1}\}$.
	The element associated to such $s = (u_0, \ldots, u_{k_0})$ in $\BS(m,n)$ is the product $\mathfrak{s}=u_1 \cdots u_{k_0}$. A \textbf{subword} $s'$ of $s$ is a sequence of the form $s'=(u_1,\ldots,u_{l_0})$, for some $l_0 \leq k_0$. By an abuse of notation, we will freely identify the word $s$ and the element $\mathfrak{s}$ of $\BS(m,n)$ (keeping in mind that the writing of $\mathfrak{s}$ as a product of elements of $\{b,b^{-1},t,t^{-1}\}$ is not unique!). In particular, the set of subwords of an element of $\BS(m,n)$ strongly depends on the chosen writing of this element as a product of elements in $\{b,b^{-1},t,t^{-1}\}$. 
	
	\begin{example}For instance, the subwords of the word $tbbt^{-1}b^{-1}b^{-1}b^{-1}$ in $\BS(2,3)$ are $\emptyset$, $t$, $tb$, $tbb$, $tbbt^{-1}$, $tbbt^{-1}b^{-1}$, $tbbt^{-1}b^{-1}b^{-1}$, $tbbt^{-1}b^{-1}b^{-1}b^{-1}$ (though this is the trivial element of $\BS(2,3)$).
	\end{example}
	
	As a particular case of \cite[Chapter 1, Section 5.2]{serre} (which yields the normal form for HNN-extensions), any element $g$ of $\BS(m,n)$ can be uniquely written as \[g = b^{n_1}t^{\varepsilon_1}b^{n_2}\cdots b^{n_r}t^{\varepsilon_r}b^{n_{r+1}}\]
	where $n_i \in \mathbb{Z}$ for every $i \in \llbracket 1, r+1 \rrbracket$, and $\varepsilon_i \in \{1,-1\}$ for every $i \in \llbracket 1,r \rrbracket$, and \begin{itemize}
		\item if $\varepsilon_i=1$, then $n_{i+1} \in \left\llbracket 0, |m|-1\right\rrbracket$;
		\item if $\varepsilon_i=-1$, then $n_{i+1} \in \left\llbracket 0, |n|-1\right\rrbracket$;
        \item there is no subword of the form $t^{\varepsilon}b^0t^{-\varepsilon}$.
	\end{itemize}
	This is the \textbf{normal form} of $g$. If $g$ is written in its normal form, we say that $g$ is \textbf{reduced}. The \textbf{height} of an element $g = b^{n_1}t^{\varepsilon_1}b^{n_2}\cdots b^{n_r}t^{\varepsilon_r}b^{n_{r+1}} \in \BS(m,n)$ written in its normal form is the integer \[\mathfrak{h}(g) = r.\]
	\begin{remark}\label{rem: dist normal form} This is also the distance \[\mathfrak{h}(g) = d_{\mathcal{T}_{m,n}}(v, g \cdot v),\] where $v = \langle b \rangle$ is the root of the Bass-Serre tree $\mathcal{T}_{m,n}$ of $\BS(m,n)$. More specifically, if $g = b^{n_1}t^{\varepsilon_1}b^{n_2}\cdots b^{n_r}t^{\varepsilon_r}b^{n_{r+1}}$ is written in its normal form, then the geodesic between $v$ and $g \cdot v$ is of the form $e_1, \ldots, e_r$, where the orientation of $e_i$ is given by the sign of $\varepsilon_i$.
	\end{remark}
	\subsubsection{Preactions and $(m,n)$-graphs}\label{projection}
	
	Any subgroup $\Lambda \leq \BS(m,n)$ acts on $\mathcal{T}_{m,n}$. As the following diagram 
	\[\xymatrix{
		{} & \Cay(\BS(m,n)) \ar[ld]_{\Lambda \backslash} \ar[rd]^{/\langle b \rangle} & {} \\
		\Sch(\Lambda) \ar[rd]_{/\langle b \rangle} & {} & \mathcal{T}_{m,n} \ar[ld]^{\Lambda \backslash} \\
		{} & \Lambda \backslash \Cay(\BS(m,n)) / \langle b \rangle & {}
	}\]
	commutes, the quotient graph $\Lambda \backslash \mathcal{T}_{m,n}$ can be obtained by shrinking all the $\langle b \rangle$-orbits of the Schreier graph of $\Lambda$ (with respect to the generating set $\{ b, b^{-1}, t, t^{-1}\}$).
	
	Motivated by this observation, we define a (right) \textbf{preaction} on a pointed countable set $(X,x_0)$ as a couple of partial bijections $(\beta, \tau)$ such that $\beta$ is a genuine bijection of $X$, $\dom(\tau)$ is $\beta^n$-invariant, $\rng(\tau)$ is $\beta^m$-invariant and $x \cdot \tau \beta^m = x \cdot \beta^n \tau$ for every $x \in \dom(\tau)$. The \textbf{Schreier $(\bm{m,n})$-graph} $\Sch(\alpha)$ of a preaction $\alpha$ on a countable set $X$ is the oriented graph whose set of vertices is $X$ and whose (positive) edges are either of the form $(x,x \cdot \beta)$, or of the form $(x, x\cdot \tau)$ (for $x \in \dom(\tau)$). Every path $c$ in $\pi_1(\Sch(\alpha),x_0)$ is labeled by a word whose letters lie in $\{b,b^{-1},t,t^{-1}\}$, thus defines an element $\psi(c) \in \Gamma$. The map $\psi : \pi_1(\Sch(\alpha),x_0) \to \Gamma$ is a group morphism, and the image of this map is called the \textbf{stabilizer} of the point $x_0$ for the preaction $\alpha$, and denoted by $\bm{\Stab_{\alpha}}\textbf{($\bm{x_0}$)}$. The \textbf{($\bm{m,n}$)-graph} $\mathcal{G}_{\alpha}$ of a preaction $\alpha$ is the Schreier $(m,n)$-graph of $\alpha$ all of whose $\beta$-orbits have been shrunk and labeled by their cardinalities. More specifically \begin{itemize}
		\item its set of vertices is $X/ \left\langle \beta \right\rangle$ and every vertex $x\left\langle \beta \right\rangle$ is labeled by the cardinality $\left|x \cdot \left\langle \beta \right\rangle\right|$;
		\item its set of positive (\textit{resp.} negative) edges is $\dom(\tau)/\left\langle \beta^n\right\rangle$ (\textit{resp.} $\rng(\tau)/\left\langle \beta^m\right\rangle$);
		\item the target and source maps are defined by $\src\left(x \cdot \left\langle \beta^n \right\rangle\right) = x \cdot \left\langle \beta \right\rangle$ and $\trg\left(x \cdot \left\langle \beta^n \right\rangle\right) = x\tau \cdot \left\langle \beta \right\rangle$. Moreover, $\overline{x \cdot \left\langle \beta^n \right\rangle} = x \tau \cdot \left\langle \beta^m \right\rangle$.
	\end{itemize}

    \begin{notation}
        Throughout the paper, when $\alpha$ is a preaction on a pointed countable set $(X, x)$, we will denote by $\mathcal{G}_{\alpha}$ the $(m,n)$-graph of $\alpha$.
    \end{notation}

    A preaction $\alpha = (\beta,\tau)$ is called \textbf{transitive} if its $(m,n)$-graph is connected. It is called \textbf{saturated} if $\dom(\beta)=\rng(\beta) = X$. If $\alpha$ is a saturated and transitive preaction on a pointed countable set $(X,x_0)$, the data of the $(m,n)$-graph of $\alpha$ is equivalent to the data of the graph of groups of $\Stab_{\alpha}(x_0)$, given by its action on the Bass-Serre tree $\mathcal{T}_{m,n}$. See \cite[Section 3]{solitar} for more details. 

     \begin{remark}\label{properinvsubtree}
        Let us assume that $\Lambda \leq \BS(m,n)$ is a finitely generated subgroup whose graph of groups (induced by its action on $\mathcal{T}_{m,n}$) is infinite (that is to say, $\Lambda \backslash \mathcal{T}_{m,n}$ is infinite). Let $\pi : \mathcal{T}_{m,n} \to \Lambda \backslash \mathcal{T}_{m,n}$ be the projection. Then, for any connected finite subgraph $\mathcal{F}$ of this graph of groups such that $\Lambda$ is the fundamental group of $\mathcal{F}$ (in the formalism of Bass Serre theory), the preimage $\pi^{-1}(\mathcal{F})$ is a proper invariant subtree of $\mathcal{T}_{m,n}$. Let us explain why. Notice that the complement of $\mathcal{F}$ in the graph of groups of $\Lambda$ is a forest; otherwise, any edge path with source and target in $\mathcal{F}$ that is not fully contained in $\mathcal{F}$ would yield an extra edge generator $t \notin \Lambda$ by Bass-Serre theory. And if $\pi^{-1}(\mathcal{F})$ was not connected, then any edge path of $\mathcal{T}_{m,n}$ connecting two distinct connected components of $\pi^{-1}(\mathcal{F})$ would be mapped by $\pi$ to an edge path with source and target in $\mathcal{F}$ that is not fully contained in $\mathcal{F}$, contradicting the previous observation.
     \end{remark}
	
	An abstract \textbf{($\bm{m,n}$)-graph} is then an oriented labeled graph all of whose vertices are labeled by an integer or $\infty$ and that satisfies the following arithmetical properties: \begin{enumerate}
		\item every vertex labeled $N$ has at most $N \wedge n$ outgoing edges and at most $N \wedge m$ incoming edges;
		\item \textbf{(Transfer Equation)} for every positive edge with source labeled $N$ and target labeled $M$, one has \begin{equation}\label{transferequation}\frac{N}{N \wedge n} = \frac{M}{M \wedge m}.\end{equation}
	\end{enumerate}
	It is called \textbf{saturated} if equalities hold for every vertex in the first item.
	
	We will make use of the following lemma that gives the existence of a particular extension of a given preaction. A proof can be found in \cite[Theorem 4.6]{solitar2}:
	
	\begin{lemma}[Maximal forest saturation]\label{maxforest}
		For any transitive and non-saturated preaction $\alpha$ on a pointed countable set $(X,x_0)$, there exists a unique transitive action $\beta$ (up to isomorphism) such that: 
\begin{itemize}
            \item $\beta$ extends $\alpha$ (in particular, $\mathcal{G}_{\beta}$ contains $\mathcal{G}_{\alpha}$ as a subgraph);
            \item $\Stab_{\beta}(x_0) = \Stab_{\alpha}(x_0)$
            \end{itemize}
            Moreover, this action has the following properties: 
            \begin{itemize}
	        \item the subgraph induced by the set of vertices of $\mathcal{G}_{\beta} \smallsetminus \mathcal{G}_{\alpha}$ is an infinite forest $\mathcal{F}$;
			\item the subgraph induced by the set of vertices of $\mathcal{G}_{\alpha}$ is $\mathcal{G}_{\alpha}$;
			\item for any edge $e \in \mathcal{E}(\mathcal{G}_{\beta}) \smallsetminus \mathcal{E}(\mathcal{G}_{\alpha})$, denoting by $N$ the label of $\src(e)$ and by $M$ the label of $\trg(e)$: \begin{itemize} \item if the half-graph of $e$ is in $\mathcal{F}$, then $M = \frac{N|m|}{N \wedge n}$;
            \item otherwise, \textit{i.e.} if the half-graph of $\overline{e}$ is in $\mathcal{F}$, then $N = \frac{M|n|}{M \wedge m}$.\end{itemize}
		\end{itemize}
	\end{lemma}

    The $\BS(m,n)$-action $\beta$ constructed in Lemma \ref{maxforest} is called the \textbf{maximal forest saturation action of $\bm{\alpha}$.}

    \begin{remark}\label{maxforestsubgroup}
        If $\Lambda$ is a finitely generated subgroup of $\BS(m,n)$, then there exists a preaction $\alpha$ on a pointed countable set $(X,x_0)$ whose $(m,n)$-graph is finite and such that $\Stab_{\alpha}(x_0) = \Lambda$. The $\BS(m,n)$-right action associated to $\Lambda$ by the correspondence \eqref{corr} is exactly the maximal forest saturation action $\beta$ of $\alpha$ given by Lemma \ref{maxforest}.
    \end{remark}

    \begin{remark}\label{invsubtreepreac}
        We keep the notations of Lemma \ref{maxforest}. Identifying $(m,n)$-graphs with graphs of subgroups, Remark \ref{properinvsubtree} tells us that, denoting by $\pi : \mathcal{T}_{m,n} \to \mathcal{G}_{\beta}$ the projection, the preimage $\pi^{-1}(\mathcal{G}_{\alpha})$ is a proper $\Stab_{\alpha}(x_0)$-invariant subtree of $\mathcal{T}_{m,n}$.
    \end{remark}

Keeping the notations of Lemma \ref{maxforest}, in the case where $|m|=|n|$ or when the $\langle b \rangle$-orbits of $\alpha$ are infinite, the forest $\mathcal{F}$ is a collection of half-subtrees of $\mathcal{T}_{m,n}$. More formally: 

\begin{proposition}\label{homeo}
Let $\alpha$ be a transitive and non-saturated preaction, and let $\beta$ be its maximal forest saturation action. Let $\mathcal{F}$ be the forest induced by the set of vertices of $\mathcal{G}_{\beta} \smallsetminus \mathcal{G}_{\alpha}$.
    Let us assume that \begin{itemize}
        \item either $|m|=|n|$;
        \item or the $\langle b \rangle$-orbits of $\alpha$ are infinite.
    \end{itemize}
   Let $f$ be any edge with source in $\mathcal{G}_{\alpha}$ and target outside $\mathcal{G}_{\alpha}$ and let $\widehat{\mathcal{T}} \subseteq \mathcal{F}$ be the half-tree of $f$. Then, the projection $\pi : \mathcal{T}_{m,n} \to \mathcal{G}_{\beta}$ induces a graph isomorphism between any connected component of $\pi^{-1}\left(\widehat{\mathcal{T}}\right)$ and $\widehat{\mathcal{T}}$. 
\end{proposition}

\begin{proof}
    By the properties of the maximal forest saturation action (\textit{cf.} Lemma \ref{maxforest}), any vertex of $\mathcal{G}_{\beta} \smallsetminus \mathcal{G}_{\alpha}$ has $|m|$ incoming edges and $|n|$ outgoing edges if $|m|=|n|$ or if the labels of the vertices of $\mathcal{G}_{\beta}$ are infinite (\textit{i.e.} the $\langle b \rangle$-orbits of $\alpha$ are infinite). 
    As any vertex of $\mathcal{T}_{m,n}$ also has $|m|$ incoming edges and $|n|$ outgoing edges, this tells us that $\pi$ induces a locally injective graph morphism between any connected component of $\pi^{-1}\left(\widehat{\mathcal{T}}\right)$ and $\widehat{\mathcal{T}}$. Finally, as $\widehat{\mathcal{T}}$ is a tree, \cite[Section 4.5, Lemma 5]{serre} tells us that this induced graph morphism is in fact an isomorphism. 
 \end{proof}

	Given a preaction $\alpha$ on a pointed countable set $(X,x_0)$, one has a projection $\Sch(\alpha) \to \mathcal{G}_{\alpha}$, \begin{itemize}
		\item that shrinks the $\langle \beta \rangle$-orbits;
		\item that sends the edge labeled $t$ connecting $x$ to $x \cdot \tau$ to the edge $x \cdot \langle \beta^n \rangle$ for every $x \in X$
			\end{itemize} (\textit{cf.} \cite[Definition 3.10]{solitar} for more details). 

    \begin{notation}
        For any preaction $\alpha$, we denote by $p_{\alpha} : \Sch(\alpha) \to \mathcal{G}_{\alpha}$ the projection.
    \end{notation}
	
	Any couple $(\gamma,x)$, where $\gamma=b^{n_1}t^{\varepsilon_1}b^{n_2}\cdots b^{n_r}t^{\varepsilon_r}b^{n_{r+1}}$ is a word in $\BS(m,n)$ and $x$ is a vertex of $\Sch(\alpha)$, leads to a unique edge path $e_1,\ldots,e_r$ in $\mathcal{G}_{\alpha}$ (whose orientation is given by the sequence of signs $(\varepsilon_1,\ldots,\varepsilon_r)$ and such that $\src(e_i) = p_{\alpha}(x \cdot b^{n_1}t^{\varepsilon_1}\cdots t^{\varepsilon_{i-1}}b^{n_{i}})$ for every $i \in \llbracket 1, r \rrbracket$, and $\trg(e_r) = p_{\alpha}(x \cdot \gamma)$). We say that the edge path $e_1,...,e_r$ \textbf{derives from} $x$ and the word $\gamma$. This observation leads to the following estimate: for any $x \in \mathcal{V}(\Sch(\alpha))$, one has \begin{equation}\label{estimate}
		d_{\mathcal{G}_{\alpha}}(p_{\alpha}(x), p_{\alpha}(x \cdot \gamma)) \leq \mathfrak{h}(\gamma).
	\end{equation}
	
	If the edge path $e_1,\ldots,e_r$ is reduced, then $\gamma$ is necessarily reduced. Conversely, if $\gamma=b^{n_1}t^{\varepsilon_1}b^{n_2}\cdots b^{n_r}t^{\varepsilon_r}b^{n_{r+1}}$ is reduced and the edge path $e_1,\ldots, e_r$ defined by $(\gamma, x)$ satisfies the following conditions: \begin{itemize}
		\item $m$ divides $|x \cdot b^{n_1}t^{\varepsilon_1}b^{n_2}\cdots b^{n_i}t^{\varepsilon_i} \langle b \rangle |$ for every $i$ such that $\varepsilon_{i} = 1$ and $\varepsilon_{i+1}=-1$;
		\item $n$ divides $|x \cdot b^{n_1}t^{\varepsilon_1}b^{n_2}\cdots b^{n_i}t^{\varepsilon_i} \langle b \rangle |$ for every $i$ such that $\varepsilon_{i}=-1$ and $\varepsilon_i=1$,
	\end{itemize} then the edge path $e_1,\ldots,e_r$ is reduced (\textit{cf.} \cite[Section 3.1, Lemma 5.9]{bontemps} for more details).
	
	From now on, we will freely identify subgroups of $\BS(m,n)$ and pointed transitive saturated preactions. 
	
	One defines a topology $\mathcal{T}_{sat}$ on the set of transitive saturated preactions as follows: a basis of neighborhoods of a transitive saturated preaction $\alpha$ on a set $(X,x)$ is given by \begin{align*}&\{\beta \text{ \ transitive saturated preaction on a pointed countable set $(X',x')$ \ } \mid \\ & (p_{\beta}^{-1}(B_{\mathcal{G}_{\beta}}(p_{\beta}(x'),R)), x') \text{ \ is isomorphic to \ } (p_{\alpha}^{-1}(B_{\mathcal{G}_{\alpha}}(p_{\alpha}(x),R)), x) \\& \text{ \ (as pointed labeled graphs)} \}\end{align*}
	(for all $R>0$).
    \begin{remark}\label{topfin}
	Notice that $\mathcal{T}_{sat}$ is strictly finer than the Chabauty topology: the subset $\mathcal{K}_{\infty}$ is open for $\mathcal{T}_{sat}$ (as it can be defined as the set of subgroups $\Lambda$ such that the preimage of the base vertex of the $(m,n)$-graph of $\Lambda$ in $\Sch(\Lambda)$ is infinite), but not for the Chabauty topology. In fact, $\mathcal{T}_{sat}$ is even strictly finer than the topology generated by the Chabauty topology and $\mathcal{K}_{\infty}$. For instance, the set \begin{align*}\mathcal{L}^o = &\big\{\Lambda \in \mathcal{K}_{\infty} \mid \text{there is no loop based at the base vertex of the $(m,n)$-graph} \\ &\text{ of $\Lambda$} \big\}\end{align*} is open for $\mathcal{T}_{sat}$. However, it is not open for the topology induced on $\mathcal{K}_{\infty}$ by the Chabauty topology. To see this, let us consider the sequence $\left(\left\langle t^2b^{mN} \right\rangle\right)_{N \in \mathbb{N}}~\in~\left(\mathcal{K}_{\infty} \smallsetminus \mathcal{L}^o\right)^{\mathbb{N}}$. If $x \in \BS(m,n)$ belongs to $\left\langle t^2b^{mN} \right\rangle $ for infinitely many $N$'s, then one can write $x = \left(t^2b^{mN}\right)^{k_N}$ where $k_N \in \mathbb{Z}$ for infinitely many $N$'s. Let us consider the morphism $\varphi : \BS(m,n) \to \mathbb{Q} \rtimes_{\frac{n}{m}} \mathbb{Z}$ (where the generator $1$ of $\mathbb{Z}$ acts on $\mathbb{Q}$ by $\frac{n}{m}$-multiplication) that sends $b$ to $(1,0)$ and $t$ to $(0,1)$. The image of $x$ under this morphism is $\varphi(x) = \left(mN \sum_{r=1}^{k_N}\left(\frac{n}{m}\right)^{2r}, 2k_N\right)$, so $k_N$ is constant (equal to some $k$). As the first coordinate $mN \sum_{r=1}^{k}\left(\frac{n}{m}\right)^{2r}$ is also constant, one necessarily has $k=0$. Thus, $x=\id$, so $\lim_{N \to \infty}\left\langle t^2b^{mN} \right\rangle~=~\{\id\} \in~\mathcal{L}^o$. We proved that $\mathcal{K}_{\infty} \smallsetminus \mathcal{L}^o$ is not closed for the topology induced by the Chabauty topology on $\mathcal{K}_{\infty}$, thus $\mathcal{L}^o$ is not open for the topology induced by the Chabauty topology and $\mathcal{T}_{sat}$.
    \end{remark}

    \begin{remark}\label{rem: basis tsat}
        Though we will mainly argue on the level of actions, one can translate this basis on the level of the set of subgroups: $\mathcal{T}_{sat}$ is then the topology generated by the basic open sets defined by $\mathcal{V}_{sat}(O,I )  = \{\Lambda \mid I \subseteq \Lambda, O \cap \Lambda = \emptyset\},$ where $O$ and $I$ are subsets of $\Gamma$ whose elements have bounded $t$-height.
    \end{remark}
	
	\subsection{Random walks on groups}\label{probas}
	
	The definitions of this section come from \cite{osin}. Let $\Gamma$ be a countable group. Let $\mu : \Gamma \to [0,1]$ be a probability measure on $\Gamma$. A \textbf{random walk on $\BS(m,n)$ with step distribution $\bm{\mu}$} is a random sequence $(G_k)_{k \in \mathbb{N}} = (S_1 \cdots S_k)_{k \in \mathbb{N}}$ of $\Gamma$ where the $S_i$'s are independently $\mu$-distributed random variables. 
	
	Let $(G_k)_{k \in \mathbb{N}}$ be a random walk on a group $\Gamma$ with step distribution $\mu$. A $\Gamma$-action on a Polish space $X$ is said \textbf{topologically $\bm{\mu}$-mixing} if for any non-empty open subsets $U, V \subseteq X$, the following holds: \[\lim_{k \to \infty}\mathbb{P}(G_k \cdot U \cap V \neq \emptyset) = 1.\]
	Notice that any topologically $\mu$-mixing action is also \textbf{$l$-topologically transitive} for every $l \in \mathbb{N}^*$: for every non-empty open subsets $U_1, \ldots , U_l, V_1, \ldots , V_l \subseteq X$, there exists an element $g \in \Gamma$ such that \[g \cdot U_i \cap V_i \neq \emptyset, \ \forall i \in \llbracket 1, l \rrbracket\]
	(see \cite[Proposition 1.2]{osin}).
	
	Given a probability measure $\mu : \Gamma \to [0,1]$, we denote by $\Supp(\mu)$ its \textbf{support}, that is to say, the set \[\Supp(\mu) = \{g \in \Gamma \mid \mu(g) > 0\}.\]
	We say that the support of $\mu$ is \textbf{generating} if $\Supp(\mu)$ generates $\Gamma$ as a semi-group, and that it is \textbf{symmetric} if $\Supp(\mu)$ is stable under inversion. 

    Given a probability measure $\mu$ on $\Gamma$, we denote by $\mu^{-1}$ the probability measure defined by: $\mu^{-1}(g) := \mu(g^{-1})$ for every $g \in \Gamma$.
	
	If $\Gamma$ acts isometrically on a metric space $X$, we say that the support of $\mu$ is \textbf{bounded} (with respect to this action) if $\Supp(\mu)(x)$ is a bounded subset of $X$ for some (equivalently for all) $x \in X$. 

    \begin{remark}\label{rem: reversed}
        If $\mu$ has one of the following properties: generating support, symmetric support, or bounded support (with respect to some isometric action of $\Gamma$), then so does $\mu^{-1}$.
    \end{remark}
	
	\begin{remark}Let $\Gamma = \BS(m,n)$, that acts on its Bass-Serre tree $\mathcal{T}_{m,n}$ and let $\mu$ be a probability measure on $\Gamma$. Then, by Remark~\ref{rem: dist normal form} the support of $\mu$ is bounded if and only if the set of integers \[\{\mathfrak{h}(g), g \in \Supp(\mu)\}\] is finite. We draw the attention of the reader to the fact that in this case, the support of $\mu$ need not be finite. 
	\end{remark}
	
	Throughout this paper, random variables are always denoted by capital letters. 
	
	\subsection{Phenotype of a subgroup of $\BS(m,n)$}
	
	In this section, we recall the computation of the perfect kernel and the construction of the decomposition of $\mathcal{K}(\BS(m,n))$ obtained in \cite{solitar}. The authors proved the following result:
	\begin{theorem}
		Let $(m,n) \in \left(\mathbb{Z} \smallsetminus \{0\}\right)^2$ such that $|m| \neq 1$ and $|n| \neq 1$. Then \[\mathcal{K}(\BS(m,n)) = \{\Lambda \leq \BS(m,n) \mid \text{ \ the $(m,n)$-graph of $\Lambda$ is infinite}\}. \]
	\end{theorem}
	
	To understand the dynamics induced by the action by conjugation, the authors introduced the notion of \textbf{phenotype}. This is a function $\Ph_{m,n} : \mathbb{N}^* \sqcup \{\infty\} \to \mathbb{N}^* \sqcup \{\infty\}$ which is constant on the set of labels of any connected $(m,n)$-graph. It is defined by \begin{equation}\label{phenotype}\Ph_{m,n}(N) :=\left\{
    \begin{array}{ll}
          \prod_{p \in \mathcal{P}, |m|_p = |n|_p \text{ \ and \ } |N|_p > |n|_p}p^{|N|_p} & \mbox{if } N \in \mathbb{N}^* \\
        \infty & \mbox{if } N = \infty.
    \end{array}
\right.
    \end{equation}
    The phenotype of a connected $(m,n)$-graph $\mathcal{G}$ is defined as follows: denoting by $N$ any label of $\mathcal{G}$, then $\bm{\Ph}_{m,n}\left(\mathcal{G}\right) = \Ph_{m,n}(N)$.

	The phenotype $\bm{\Ph}_{m,n}(\Lambda)$ (\textit{resp.} $\bm{\Ph}_{m,n}(\alpha)$) of a subgroup $\Lambda \leq \BS(m,n)$ (\textit{resp.} of a preaction $\alpha$) is then the phenotype of its $(m,n)$-graph. As the labels of the $(m,n)$-graph of $\Lambda$ encode the intersections of the conjugates of $\Lambda$ with $\langle b \rangle$, one also has $\bm{\Ph}_{m,n}(\Lambda) = \Ph_{m,n}\left(\left[\langle b \rangle : \Lambda \cap \langle b \rangle\right]\right)$ (as defined in \eqref{phenotype}). 
	
	The authors of \cite{solitar} proved that this quantity is invariant under conjugation and that the set $\mathcal{Q}_{m,n} = \bm{\Ph}_{m,n}(\Sub(\Gamma))$ is infinite. Thus, we get an infinite countable invariant partition of $$\Sub(\BS(m,n)) = \bigsqcup_{P \in \mathcal{Q}_{m,n}}\bm{\Ph}_{m,n}^{-1}(P),$$ which yields a partition of the perfect kernel \[\mathcal{K}(\BS(m,n)) = \bigsqcup_{P \in \mathcal{Q}_{m,n}}\mathcal{K}(\BS(m,n)) \cap \bm{\Ph}_{m,n}^{-1}(P).\]
    \begin{notation}
        For any $P \in \mathcal{Q}_{m,n}$, we denote by $\mathcal{K}_P$ the intersection $\mathcal{K}(\BS(m,n)) \cap \bm{\Ph}_{m,n}^{-1}(P) = \{\Lambda \in\bm{\Ph}_{m,n}^{-1}(P) \mid \Lambda \backslash \mathcal{T}_{m,n} \text{ is infinite}\}$.
    \end{notation}
    Moreover, they proved the following result:
	
	\begin{theorem}
		Let $(m,n) \in \left(\mathbb{Z} \smallsetminus \{0\}\right)^2$ such that $|m| \neq 1$ and $|n| \neq 1$. Then, for any $P \in \mathcal{Q}_{m,n}$: \begin{itemize}
        	\item if $P \in \mathbb{N}$, the piece $\bm{\Ph}_{m,n}^{-1}(P)$ is open;
			\item it is also closed if and only if $|m|=|n|$;
			\item the piece $\bm{\Ph}_{m,n}^{-1}(\infty)$ is closed;
			\item $\mathcal{K}_P$ is non-empty and the action by conjugation of $\BS(m,n)$ on it is topologically transitive.
		\end{itemize}  
	\end{theorem}
	In \cite{solitar2}, the same authors even proved that the action of $\BS(m,n)$ by conjugation on $\mathcal{K}_P$ is highly topologically transitive if $P \in \mathcal{Q}_{m,n}$.
	
	\begin{remark}\label{labelunimod} In the case where $|m|=|n|$, it is not hard to check that the phenotype of a non-zero integer $N$ satisfies the following equality: \[\Ph_{\pm n,n}(N)\left(\prod_{p \in \mathcal{P}, \left|\Ph_{\pm n,n}(N)\right|_p=0}p^{|n|_p}\right) = \frac{Nn}{N \wedge n}.\] 
    
    In particular, if $\alpha$ is a transitive and non-saturated preaction, denoting by $\beta$ its maximal forest saturation action, any vertex of $\mathcal{G}_{\beta} \smallsetminus \mathcal{G}_{\alpha}$ is labeled $\bm{\Ph}_{\pm n,n}(\mathcal{G}_{\alpha})\left(\prod_{p \in \mathcal{P}, \left|\bm{\Ph}_{\pm n,n}(\mathcal{G}_{\alpha})\right|_p=0}p^{|n|_p}\right)$.
    \end{remark}
	
	\section{Random walks on the Bass-Serre tree and on $(m,n)$-graphs}
	
	Let $m,n \in \mathbb{Z}$ such that $\min(|m|, |n|) > 1$. Let $\mathcal{T}_{m,n}$ be the Bass-Serre tree of the Baumslag-Solitar group $\Gamma := \BS(m,n)$ of parameters $(m,n)$. 
	
	\subsection{Proof of Theorem \ref{nonmixing}}
	
	In this section, we assume that $|m| \neq |n|$. The goal is to build a probability measure $\mu$ whose support is $\{b,b^{-1},t,t^{-1}\}$ such that, for every finite $P \in \mathcal{Q}_{m,n}$, the action by conjugation on $\mathcal{K}_P$ is not topologically $\mu$-mixing. 
	
	To prove Theorem \ref{nonmixing}, we will consider a probability measure supported on $\{b,b^{-1},t,t^{-1}\}$ satisfying $\mu(t) \neq \mu\left(t^{-1}\right)$. We will use the following deterministic result:
	
	\begin{proposition}\label{evanescent}
		Let $s_1, \ldots, s_r \in \{b,b^{-1},t,t^{-1}\}$. For every $i \in \llbracket 1,r \rrbracket$, let \[\mathfrak{h}_i^+ := \left|\{j \in \llbracket 1,i\rrbracket \mid s_j = t\}\right|\] and \[\mathfrak{h}_i^- = \left|\{j \in \llbracket 1,i\rrbracket \mid s_j = t^{-1}\}\right|\] (and $\mathfrak{h}_0^+ = \mathfrak{h}_0^- = 0$). Let $\alpha$ be a preaction on a pointed countable set $(X,x)$ and $p$ be a prime number such that \begin{itemize} 
        \item $|m|_p > |n|_p$;
        \item $\mathfrak{h}_i^+ > \mathfrak{h}_i^-$ for every $i \in \llbracket 1, r \rrbracket$;
        \item $b$ and $s_1$ are defined on $x$, and for every $i \in \llbracket 2, r \rrbracket$, the elements $b$ and $s_i$ are defined on $x \cdot s_1 \cdots s_{i-1}$;
        \item the cardinal $N$ of the $\langle b \rangle$-orbit of $x$ satisfies $|N|_p > |m|_p$.
        \end{itemize}

		Then, for every $i \in \llbracket 0, r \rrbracket$, the cardinality $N_i$ of the $\langle b \rangle$-orbit of $x \cdot s_1 \cdots s_i$ satisfies \[|N_i|_p = (\mathfrak{h}_i^+-\mathfrak{h}_i^-)(|m|_p-|n|_p) + |N|_p\]
        (where $N_0=N$ and $x \cdot s_1 \cdots s_i=x$ if $i=0$).
	\end{proposition}
	
	\begin{proof}
		We proceed by induction on $i \in \llbracket 0, r \rrbracket$. 
		
		\paragraph{\textbf{Base case}} The result is clear for $i=0$.
		\paragraph{\textbf{Induction step}} Let us assume that $i \in \llbracket 0, r-1 \rrbracket$ and that \[|N_i|_p = (\mathfrak{h}_i^+-\mathfrak{h}_i^-)(|m|_p-|n|_p) + |N|_p.\]
		We distinguish three cases:
		
		\subparagraph{\textbf{1st case: $s_{i+1} \in \{b,b^{-1}\}$}} In this case, the $\langle b \rangle$-orbits of $x \cdot s_1 \cdots s_i$ and $x \cdot s_1 \cdots s_{i+1}$ are the same. Thus, $N_{i+1} = N_i$. Moreover, $\mathfrak{h}_{i+1}^+ = \mathfrak{h}_i^+$ and $\mathfrak{h}_{i+1}^- = \mathfrak{h}_i^-$, thus $|N_{i+1}|_p = (\mathfrak{h}_{i+1}^+-\mathfrak{h}_{i+1}^-)(|m|_p-|n|_p)+|N|_p$.\\ 
		
		\subparagraph{\textbf{2nd case: $s_{i+1} = t$}} In this case, one has $\mathfrak{h}_{i+1}^+ = \mathfrak{h}_i^++1$ and $\mathfrak{h}_{i+1}^- = \mathfrak{h}_i^-$. By the Transfer Equation \eqref{transferequation}, one has $\frac{N_i}{N_i \wedge n} = \frac{N_{i+1}}{N_{i+1} \wedge m}$, which implies that \begin{equation}\label{transfert}
			|N_i|_p - \min(|N_i|_p, |n|_p) = |N_{i+1}|_p - \min(|N_{i+1}|_p,|m|_p).
		\end{equation}
		By the induction hypothesis \begin{align*}
			|N_i|_p &= (\mathfrak{h}_i^+-\mathfrak{h}_i^-)(|m|_p-|n|_p) + |N|_p \\
			&\geq |N|_p \text{ \ by the assumption made on $\mathfrak{h}_i^+, \mathfrak{h}_i^-$ and the assumption made on $p$} \\
			&> |m|_p \text{ \ by the assumption made on $N$} \\
			&> |n|_p \text{ \ by the assumption made on $p$}.
		\end{align*}
		Thus by Equation \eqref{transfert}: \begin{align*}
			|N_{i+1}|_p - \min(|N_{i+1}|_p,|m|_p) &= |N_i|_p - |n|_p \\
			&> 0
		\end{align*}
		so necessarily \begin{align*}
			|N_{i+1}|_p &= |N_i|_p+|m|_p-|n|_p \\
			&= (\mathfrak{h}_i^+-\mathfrak{h}_i^-+1)(|m|_p-|n|_p) + |N|_p \\
			&= (\mathfrak{h}_{i+1}^+-\mathfrak{h}_{i+1}^-)(|m|_p-|n|_p) + |N|_p,
		\end{align*}
		which concludes the induction step in the case where $s_{i+1} = t$. \\
		
		\subparagraph{\textbf{3rd case: $s_{i+1} = t^{-1}$}} In particular, $\frac{N_i}{N_i \wedge m} = \frac{N_{i+1}}{N_{i+1} \wedge n}$ and $\left(\mathfrak{h}_{i+1}^+,\mathfrak{h}_{i+1}^-\right) = \left(\mathfrak{h}_{i}^+,\mathfrak{h}_{i}^-+1\right)$. This last case is very similar, so we leave it to the reader. 
	\end{proof}
	
	Before proving Theorem \ref{nonmixing}, let us recall some basic facts about one dimensional random walks. 
	
	\begin{proposition}\label{randomwalk}
		Let $(X_i)_{i \in \mathbb{N}}$ be a sequence of independently and identically distributed random variables, valued in $\{-1,0,1\}$, such that $\mathbb{P}(X_i=1) > \mathbb{P}(X_i=-1)$. Let us denote by $Z_n := \sum_{i=1}^nX_i$ (where $Z_0=0$), and by $p_+ := \mathbb{P}(X_i=1)$ and $p_- := \mathbb{P}(X_i=-1)$. Then \begin{enumerate}
			\item $\mathbb{P}\left(\lim_{n \to \infty}\frac{1}{n}Z_n = p_+-p_-\right) = 1$;
			\item $\mathbb{P}(Z_n > 0, \ \forall n > 0) = p_+-p_-$.
		\end{enumerate}
	\end{proposition}
	
	\begin{proof}
		
		The first item results from the strong law of large numbers. 
		
		For the second one, notice that, denoting by $p_{k,l}$ the probability that $Z_n$ reaches $l$, starting from $k > l$, one has \begin{align*}
			p_{1,0} &= \mathbb{P}(X_1=-1) + \mathbb{P}(X_1=0)p_{1,0} + \mathbb{P}(X_1=1)p_{2,0}
		\end{align*}
		As any realization of $(Z_n)_{n \in \mathbb{N}}$ starting from $2$ and reaching $0$ has to reach $1$ before $0$, and  $(Z_n)_{n \in \mathbb{N}}$ is a Markov chain, one has $p_{2,0} = p_{2,1}p_{1,0}$. Noticing that $p_{k,l} = p_{k-l,0}$ for any $k > l$, one deduces that $p_{1,0}$ satisfies the following equation:
		\[p_+p_{1,0}^2 - (p_++p_-)p_{1,0} + p_- = 0\]
		which leads to $p_{1,0} = \frac{p_-}{p_+}$ or $p_{1,0} = 1$. Thus \begin{equation}\label{pos}\begin{split}
			\mathbb{P}(Z_n > 0, \ \forall n > 0) &= \mathbb{P}(X_1=1)(1-p_{1,0}) \\
			&= \left\{
    \begin{array}{ll}
        p_+-p_- & \mbox{if } p_{1,0} = \frac{p_-}{p_+} \\
        0 & \mbox{if } p_{1,0} = 1.
    \end{array}
\right.\end{split}
		\end{equation}
        We want to show that $\mathbb{P}(Z_n > 0, \ \forall n > 0)=p_+-p_-$.
        The first item of the proposition implies that \[\mathbb{P}(\exists n_0 \in \mathbb{N}, Z_{n_0}=0 \text{ \ and \ } Z_n > 0, \forall n>n_0) = 1.\] In particular, there exists $n_0 \in \mathbb{N}$ such that \[\mathbb{P}(Z_{n_0}=0 \text{ \ and \ } Z_n > 0, \forall n>n_0) > 0.\]
        In particular, \begin{align*}
            \mathbb{P}(Z_n >0, \forall n>0) &= \mathbb{P}\left(Z_n >0, \forall n>n_0 \mid Z_{n_0}=0\right) \\
            &= \frac{\mathbb{P}(Z_{n_0}=0 \text{ \ and \ } Z_n > 0, \forall n>n_0)}{\mathbb{P}(Z_{n_0}=0)} \\
            &>0.
        \end{align*}
        Thus, Equation \eqref{pos} implies that \[\mathbb{P}(Z_n > 0, \ \forall n > 0) = p_+-p_-,\] which concludes the proof. \end{proof}
	
	We are now ready to prove Theorem \ref{nonmixing}, which is a particular case of the following theorem: 
	\begin{theorem}
		Let us assume that $|m|\neq|n|$. Let $p$ be a prime number such that $|m|_p \neq |n|_p$.
		Let $\mu : \Gamma \to [0,1]$ be a probability measure that satisfies the following properties: \begin{itemize}
			\item $\Supp(\mu) = \{b, b^{-1}, t, t^{-1}\}$;
			\item $\mu(t) > \mu(t^{-1})$ if $|m|_p > |n|_p$;
			\item $\mu(t) < \mu(t^{-1})$ if $|m|_p < |n|_p$.
		\end{itemize}
		Let $P \in \mathcal{Q}_{m,n} \cap \mathbb{N}$. Then, the action by conjugation on $\mathcal{K}_P$ is not topologically $\mu$-mixing.
	\end{theorem}
	
	\begin{proof}
		Up to exchanging $m$ and $n$, let us assume that $|m|_p > |n|_p$.
		Let $N \in \Ph_{m,n}^{-1}(P)$ such that $|N|_p > |m|_p$. Such $N$ exists, because $\Ph_{m,n}(Np^r) = \Ph_{m,n}(N)$ for every $r \in \mathbb{N}$. Let $G_k = S_1 \cdots S_k$ be the random walk with step distribution $\mu$. Let us define a sequence of random variables $(X_i)_{i \in \mathbb{N}}$, valued in $\{-1,0,1\}$, as follows: 
		\[X_i = \left\{
		\begin{array}{ll}
			1 & \mbox{if \ } S_i=t;\\
			-1 & \mbox{if \ } S_i=t^{-1}; \\
			0 & \mbox{otherwise,} 
		\end{array}
		\right.\]
		and notice that the variables $X_i$ are iid. Their law is given by \begin{itemize}
			\item $\mathbb{P}(X_i=1) = \mu(t)$;
			\item $\mathbb{P}(X_i=-1) = \mu\left(t^{-1}\right)$;
			\item $\mathbb{P}(X_i=0) = \mu(b)+\mu(b^{-1}) =1 - \left(\mu(t)+\mu\left(t^{-1}\right)\right).$
		\end{itemize}
		For every $k \in \mathbb{N}$, one has:
		\begin{align*}
			\sum_{i=1}^kX_i &= \sum_{i=1}^k {\mathbb{1}}_{\{S_i=t\}} - \mathbb{1}_{\{S_i=t^{-1}\}} \\
			&= \left|\{i \in \llbracket 1,k \rrbracket \mid S_i = t\}\right| - \left|\{i \in \llbracket 1,k \rrbracket \mid S_i = t^{-1}\}\right|.
		\end{align*}
		Thus, by Lemma \ref{randomwalk} applied to the sequence $(X_i)_{i \in \mathbb{N}}$, we get the two following equalities: \begin{enumerate} \item \begin{align*}\mathbb{P}\left(\left|\{i \in \llbracket 1,k \rrbracket \mid S_i = t\}\right| > \left|\{i \in \llbracket 1,k \rrbracket \mid S_i = t^{-1}\}\right|, \ \forall k \in \mathbb{N} \right)&= \mathbb{P}\left(\sum_{i=1}^kX_i > 0, \forall k \in \mathbb{N}\right) \\
			&=\mu(t) - \mu\left(t^{-1}\right) \\ &> 0.\end{align*}
		\item \begin{align*}&\mathbb{P}\left(\lim_{k \to \infty}\frac{\left|\{i \in \llbracket 1,k \rrbracket \mid S_i = t\}\right| - \left|\{i \in \llbracket 1,k \rrbracket \mid S_i = t^{-1}\}\right|}{k} = \mu(t) - \mu\left(t^{-1}\right)\right)\\
			&=\mathbb{P}\left(\lim_{k \to \infty}\frac{\sum_{i=1}^kX_i}{k} = \mu(t) - \mu\left(t^{-1}\right)\right)\\
			&=1.\end{align*} \end{enumerate}
		For now, let us argue deterministically. Let $(s_k)_{k \in \mathbb{N}} \in \{b,b^{-1},t,t^{-1}\}$ be a sequence of elements satisfying these conditions, \textit{i.e.} with the notations of Proposition \ref{evanescent}:
		\begin{enumerate}
			\item $\mathfrak{h}_k^+ > \mathfrak{h}_k^-$ for every $k \in \mathbb{N}^*$;
			\item $\lim_{k \to \infty}\frac{\mathfrak{h}_k^+-\mathfrak{h}_k^-}{k} = \mu(t) - \mu\left(t^{-1}\right)$
		\end{enumerate}
		and let $g_k = s_1 \cdots s_k$.
		For any subgroup $\Lambda \leq \Gamma$ such that $\Lambda \cap \langle b \rangle = \left\langle b^N \right\rangle$, one has $g_k^{-1} \Lambda g_k \cap \langle b \rangle = \left\langle b^{N_k} \right\rangle$ where \begin{align*}|N_k|_p &= (\mathfrak{h}_k^+-\mathfrak{h}_k^-)(|m|_p-|n|_p)+|N|_p \text{ \ by the first condition and Proposition \ref{evanescent}}\\
			&> k\frac{\mu(t)-\mu\left(t^{-1}\right)}{2} \text{ \ for $k \geq k_0$ large enough by the second condition} \end{align*}
		(where $k_0$ only depends on $(s_k)_{k \in \mathbb{N}}$). 

        For any $M \in \Ph^{-1}_{m,n}(P)$, let us introduce the open subset $U_M = \left\{\Lambda \in \mathcal{K}(\Gamma) \mid \Lambda \cap \langle b \rangle = \left\langle b^M\right\rangle\right\}$ of $\mathcal{K}_P$ (which is non-empty, because it contains $\left\langle b^M \right\rangle$). One has \[g_k^{-1}U_Ng_k \cap U_M = \emptyset \text{ \ as soon as $k > \max\left(k_0, \frac{2|M|_p}{\mu(t) - \mu\left(t^{-1}\right)}\right)$}.\]
		
		Thus, for any $M, N \in \Ph_{m,n}^{-1}(P)$ such that $|N|_p > |m|_p$:
		\[\mathbb{P}\left( \exists k_0: G_k^{-1}U_NG_k \cap U_M = \emptyset, \ \forall k \geq k_0\right)  \geq \mu(t) - \mu\left(t^{-1}\right),\]
		which is a strong negation of being topologically $\mu$-mixing for the action by conjugation on $\bm{\Ph}_{m,n}^{-1}(P)$. 
	\end{proof}
	
	\subsection{Proof of Theorem \ref{mixing}}
	
	In this section, we consider a probability measure $\mu : \Gamma \to [0,1]$ whose support is bounded, symmetric, and generates $\Gamma$. Throughout the section, we denote by $(G_k)_{k \in \mathbb{N}} = (S_1 \cdots S_k)_{k \in \mathbb{N}}$ the random walk on $G$ with step distribution $\mu$. The ultimate goal will be the proof of Theorem~\ref{mixing}. Namely, starting from two finitely generated subgroups $\Lambda_1, \Lambda_2$ such that \begin{itemize}
	    \item either $P := \bm{\Ph}_{m,n}(\Lambda_i) = \infty$ for $i \in \{1,2\}$;
        \item or $|m|=|n|$ and $P:= \bm{\Ph}_{m,n}(\Lambda_1) = \bm{\Ph}_{m,n}(\Lambda_2)$,
	\end{itemize}
and fixing a finite subgraph $F_i$ of $\Sch(\Lambda_i)$ that contains the base point $x_i$ (for $i \in \{1,2\}$), then, with probability tending to $1$ as $k$ tends to $\infty$, we will prove the existence of some $\Lambda \in \mathcal{K}_P$ \begin{itemize}
    \item whose Schreier graph contains $F_1$ and $F_2$;
    \item such that, in the Schreier graph of $\Lambda$, one has $x_1 \cdot S_1 \cdots S_k = x_2$.
\end{itemize}
We now explain the outline of the proof. First of all, for $i \in \{1,2\}$, we consider the smallest subgraph $F_i^{sat}$ of $\Sch(\Lambda_i)$ that contains $F_i$, and that is the $(m,n)$-Schreier graph of some preaction $\alpha_i$ (namely, all $\langle b \rangle$-orbits are saturated, and, whenever $t$ (\textit{resp.} $t^{-1}$) is defined on some point, it is defined on its whole $\langle b^m \rangle$-orbit (\textit{resp.} $\langle b^n \rangle$-orbit)). In particular, $\mathcal{G}_{\alpha_i}$ is finite. Then, the proof works in two steps: 
\paragraph{\textbf{1st step}: Escaping the compact core (Subsection~\ref{subsec: escaping})} We prove that, almost surely, there exists some $k_0 \in \mathbb{N}$ such that, for all $k > k_0$, the projection of $x_1 \cdot G_k$ in the $(m,n)$-graph of the maximal forest saturation $\beta_1$ of $\alpha_1$ remains very far from $\mathcal{G}_{\alpha_1}$ (Corollary~\ref{convloinquotient}). To achieve this, we prove that \begin{enumerate}
    \item with non-zero probability, the projection of $x_1 \cdot G_k$ escapes $\mathcal{G}_{\alpha_1}$ (Lemma~\ref{sort});
    \item once it is far enough from $\mathcal{G}_{\alpha_1}$, the probability that it remains outside forever is bounded below by a positive number (Lemma~\ref{lem: borel cantelli}). This is where we use the main result of \cite{end}; namely, the fact that the random walk applied to any vertex of $\mathcal{T}_{m,n}$ converges almost surely to a random end of $\mathcal{T}_{m,n}$. 
\end{enumerate}

\paragraph{\textbf{2nd step}: Pasting preactions (Subsection \ref{subsec: paste})} The same argument applied to $\alpha_2$ (and its maximal forest saturation $\beta_2$) and the reversed random walk (\textit{i.e.} the random walk on $\Gamma$ with step distribution $\mu^{-1}$) leads to some integers $k_0, k_1, k_2 \in \mathbb{N}$ such that $k_0 > k_1+k_2$ and, for all $k > k_0$ \begin{itemize}
    \item the projection of $x_1 \cdot S_1 \cdots S_{i}$ in $\mathcal{G}_{\beta_1}$ remains very far from $\mathcal{G}_{\alpha_1}$ for every $i \in \llbracket k_1+1, k \rrbracket$;
    \item the projection of $x_2 \cdot S_k^{-1} \cdots S_{i}^{-1}$ in $\mathcal{G}_{\beta_2}$ remains very far from $\mathcal{G}_{\alpha_2}$ for every $i \in \llbracket 1, k-k_2 \rrbracket$;
    \item the distance between the projections of $x_1$ and of $x_1 \cdot S_1 \cdots S_{k-k_2}$ in $\mathcal{G}_{\beta_1}$ is very large
\end{itemize}
with high probability. For any realization $(s_i)_{i \in \mathbb{N}}$ of the random walk that satisfies the three previous conditions, denoting by $\beta_1'$ (\textit{resp.} $\beta_2'$) the subpreaction of $\beta_1$ (\textit{resp.} $\beta_2$) induced on $F_1^{sat}$ $\cup$ $\{p_{\beta_1}^{-1}(p_{\beta_1}(x_1 \cdot u)) \mid u \text{ subword of $s_1 \cdots s_{k_1}$ }\}$ (\textit{resp.} $F_2^{sat} \cup \left\{p_{\beta_2}^{-1}(p_{\beta_2}(x_2 \cdot u)) \mid u \text{ subword of}\right.$ \\ $\left. s_k^{-1} \cdots s_{k-k_2+1}^{-1}\right\}$)  we build a preaction that extends both $\beta_1'$ and $\beta_2'$, and such that, denoting by $y_1 = x_1 \cdot s_1 \cdots s_{k_1}$ (\textit{resp.} $y_2 = x_2 \cdot s_k^{-1} \cdots s_{k-k_2+1}^{-1}$), one has \[y_1 \cdot s_{k_1+1} \cdots s_{k-k_2} = y_2.\] 
While defining this preaction, one has to be careful with cancellations $tt^{-1}$, $tb^mt^{-1}b^{-n}$, \textit{etc.} The three assumptions ensure that there are not too many such cancellations. This is achieved in Lemma~\ref{merge}.

	\subsubsection{Escaping the compact core}\label{subsec: escaping} The goal of this subsection is the proof of Corollary~\ref{convloinquotient}. We first argue on the level of the Bass-Serre tree. Once and for all, we denote by $v := \langle b \rangle$ the base vertex of $\mathcal{T}_{m,n}$.

    The following lemma tells us that the random walk escapes any ``small enough'' subtree of $\mathcal{T}_{m,n}$ almost surely.
	
	\begin{lemma}\label{sort}
		Let $\Lambda$ be a subgroup of $\Gamma$ that acts cocompactly on a proper subtree $\mathcal{T}^{(0)}$ of $\mathcal{T}_{m,n}$. Then, for every $w \in \Gamma$: \[\mathbb{P}\left(wG_k \cdot v \in \mathcal{T}^{(0)} \ \forall k \in \mathbb{N}\right) = 0. \]
	\end{lemma} 

\begin{remark}
    By Bass-Serre theory, the existence of $\mathcal{T}^{(0)}$ is equivalent to the fact that $\Lambda$ is finitely generated and $\Lambda \backslash \mathcal{T}_{m,n}$ is infinite (see also Remark~\ref{properinvsubtree}).
\end{remark}

\begin{proof}
    We first build a finite set $F \subseteq \Gamma$ satisfying the following property: for every $w \in \Gamma$, there exists $f_{w} \in F$ such that $w f_{w} \cdot v \notin \mathcal{T}^{(0)}$. 

    To achieve this, using the cocompactness of $\Lambda \curvearrowright \mathcal{T}^{(0)}$, let us fix a finite set $F_0 \subseteq \Gamma$ such that \begin{equation}\label{eq: lambda cocpct}\Lambda F_0 \cdot v = \mathcal{V}\left(\mathcal{T}^{(0)}\right).\end{equation}
    As $\mathcal{T}^{(0)}$ is a proper subtree of $\mathcal{T}_{m,n}$, for any $f \in F_0$ we also fix $\gamma_f \in \Gamma$ such that $f\gamma_f \cdot v \notin \mathcal{T}^{(0)}$. 
    
    Now let us take any $w \in \Gamma$. If $w \cdot v \notin \mathcal{T}^{(0)}$, then taking $f_w = 1$ implies that $w f_w \cdot v \notin \mathcal{T}^{(0)}$. Otherwise, by Equation~\eqref{eq: lambda cocpct}, there exists some $f \in F_0$ and some $\lambda \in \Lambda$ such that $w \cdot v = \lambda f \cdot v$. As $\Stab(v) = \langle b \rangle$, there exists some $l_0 \in \mathbb{Z}$ such that $w b^{l_0} = \lambda f $. Let $h_f$ be the height of $\gamma_f$ and let \[N_f = |mn|^{h_f}.\] Writing the euclidean division of $l_0$ by $N_f$, one gets $l_0 = - qN_f + k_0$ for some $q \in \mathbb{Z}$ and some $k_0 \in \llbracket 0, N_f-1 \rrbracket$. In particular, $wb^{k_0} = \lambda f b^{qN_f}$. As $h_f$ is the height of $\gamma_f$, the defining relation of $\BS(m,n)$ implies that $b^{qN_f}\gamma_f \in \gamma_f \langle b \rangle$, thus $wb^{k_0}\gamma_f \cdot v = \lambda f \gamma_f \cdot v$. As $f \gamma_f \cdot v \notin \mathcal{T}^{(0)}$ and $\mathcal{T}^{(0)}$ is $\Lambda$-invariant, we finally get that $wb^{k_0}\gamma_f \cdot v \notin \mathcal{T}^{(0)}$.
    
    So $$F = \{1\} \cup \bigcup_{f \in F_0}\left\{b^{k_0}\gamma_f \mid k_0 \in \llbracket 0, N_f-1\rrbracket\right\}$$ is suitable.

		For every $f \in F$, using the fact that $\Supp(\mu)$ is a symmetric and generating set of $\Gamma$, let us write $f = \mathfrak{g}_1^f \cdots  \mathfrak{g}_{N_f}^f$ for some $N_f \in \mathbb{N}$ and $\mathfrak{g}_i^f \in \Supp(\mu)$ for every $i \in \llbracket 1, N_f \rrbracket$. Let us denote by \[\theta = \min_{f \in F}\prod_{i=1}^{N_f}\mu\left(\mathfrak{g}_i^f\right) \in ]0,1[\]
		and by \[L = \max_{f \in F}N_f.\]
        
        For every $N \in \mathbb{N}^*$ and $w \in \Gamma$, let us define the event \[A_{N,w} := \{ w G_k \cdot v \in \mathcal{T}^{(0)} \ \forall k \in \llbracket 0, NL-1 \rrbracket\}. \] 

    For any $N \in \mathbb{N}$ and $w \in \Gamma$, observe that:
\begin{align*}
    \mathbb{P}(A_{N+1,w}) &= \sum_{g \in \Gamma \mid  w g \cdot v \in \mathcal{T}^{(0)}}\mathbb{P}(A_{N,w} \cap \{G_{NL-1} = g\} \text{ and }  wgS_{NL} \cdots S_k \cdot v \in \mathcal{T}^{(0)} \\ & \ \ \ \ \  \ \ \ \ \ \ \ \ \ \ \ \ \ \ \ \ \ \ \ \ \ \ \ \ \ \ \ \ \ \ \ \ \ \ \ \ \ \ \ \ \ \ \ \ \ \ \ \ \ \ \ \ \ \ \ \ \ \ \ \ \ \ \ \ \ \ \  \forall k \in \llbracket  NL, (N+1)L-1 \rrbracket ) \\
            &= \sum_{g \in \Gamma \mid wg \cdot v \in \mathcal{T}^{(0)}}\mathbb{P}(A_{N,w} \cap \{G_{NL-1} = g\}) \cdot \mathbb{P}(wgS_{NL} \cdots S_k \cdot v \in \mathcal{T}^{(0)} \\& \ \ \ \ \ \ \ \ \ \ \ \ \ \ \  \  \ \ \ \ \ \ \ \ \ \ \ \ \ \ \ \ \ \ \ \ \ \ \  \ \ \ \ \ \ \ \ \ \ \ \ \ \ \ \ \ \ \ \ \ \ \ \ \ \ \ \ \ \ \ \ \ \forall k \in \llbracket  NL, (N+1)L-1 \rrbracket ),
		\end{align*}
        the last equality resulting from the independence of the $S_i$'s. Thus \begin{align*}
             \mathbb{P}(A_{N+1,w}) &\leq \sum_{g \in \Gamma \mid wg \cdot v \in \mathcal{T}^{(0)}}\mathbb{P}(A_{N,w} \cap \{G_{NL-1} = g\})\cdot \mathbb{P}(S_{NL} \cdots S_{NL + N_{f_{wg}}-1} \neq f_{wg} ) \\
            &\leq (1-\theta) \sum_{g \in \Gamma \mid wg \cdot v \in \mathcal{T}^{(0)}}\mathbb{P}(A_{N,w} \cap \{G_{NL-1} = g\})  \\
            &= (1-\theta)\mathbb{P}(A_{N,w}).
        \end{align*}

Thus, a straightforard induction implies that $\mathbb{P}(A_{N, w}) \leq (1-\theta)^N$, which tends to $0$ as $N$ tends to $\infty$. 
\end{proof}

In fact, a stronger statement holds: in the setting of the previous lemma, the random walk $G_k \cdot v$ escapes $\mathcal{T}^{(0)}$ \textit{and never returns} almost surely. This is encoded by the following proposition: 

\begin{proposition}\label{convloin}
    Let $\Lambda$ be a subgroup of $\Gamma$ that acts cocompactly on a proper subtree $\mathcal{T}^{(0)}$ of $\mathcal{T}_{m,n}$. Then, the random walk $(G_k \cdot v)_{k \in \mathbb{N}}$ converges almost surely to a random end $\xi \in \partial \mathcal{T}_{m,n} \smallsetminus \partial \mathcal{T}^{(0)}$.
\end{proposition}

	To prove Proposition~\ref{convloin}, we will make use of the following result, which is a consequence of the main theorem of \cite{end}:
	
	\begin{theorem}\label{conv}
		Let $(m,n)\in\mathbb{Z}^2$ such that $\min\left(|m|,|n|\right)>1$. Then, the sequence $(G_k \cdot v)_{k \in \mathbb{N}}$ converges almost surely to a random end $\xi \in \partial \mathcal{T}_{m,n}$. 
	\end{theorem}
	
	Theorem \ref{conv} relies on the fact that $(G_k)_{k \in \mathbb{N}}$ is a regular random walk on the automorphism group $\mathrm{Aut}(\mathcal{T}_{m,n})$ of the infinite locally finite tree $\mathcal{T}_{m,n}$ and that $\Supp(\mu)$ is not contained in any amenable closed subgroup of $\mathrm{Aut}(\mathcal{T}_{m,n})$. Thus, the hypotheses of the main theorem of \cite{end} are satisfied. For more details, see \cite[Lemma 4.8]{rwbsmn}. 

Hence, Theorem~\ref{conv} gives the existence of $\xi \in \partial \mathcal{T}_{m,n}$ in Proposition~\ref{convloin}, so what we need to prove is that, almost surely, $\xi \notin \partial \mathcal{T}^{(0)}$. Before entering into technical details, let us give an intuition about the proof. Lemma~\ref{sort} implies that the random walk starting from any element $w$ of $\Gamma$ escapes $\mathcal{T}^{(0)}$ almost surely. We will prove that, when it reaches a vertex outside a suitable neighborhood of $\mathcal{T}^{(0)}$, the probability of never returning into $\mathcal{T}^{(0)}$ is bounded below by a strictly positive probability (Lemma~\ref{lem: borel cantelli}). Then, a standard argument using stopping times will show that, in fact, the random walk escapes $\mathcal{T}^{(0)}$ and never returns almost surely, which ensures that $\xi \notin \partial \mathcal{T}^{(0)}$ almost surely. Here, the main subtlety comes from the fact that the sequence $(S_1 \cdots S_k \cdot v)_{k \in \mathbb{N}}$ is not a Markov chain: because of the non-triviality of the vertex stabilizers (for the action of $\Gamma$ on $\mathcal{T}_{m,n}$), the image $S_1 \cdots S_{k+1} \cdot v$ does not only depend on $S_1 \cdots S_{k} \cdot v$, but on the element $S_1 \cdots S_k$ of $\Gamma$. Thus, when bounding probabilities by below, we have to condition on the realization of the random walk in the group at a previous step, not on the image of this realization in $\mathcal{T}_{m,n}$.

Now we give some technical lemmas to prove Proposition~\ref{convloin}. Before that, let us introduce some notations. Using the fact that $\Supp(\mu)$ is bounded, let us denote by $$M := \max_{\gamma \in \Supp(\mu)}\mathfrak{h}(\gamma).$$ Let $F_0$ be the (finite) set of reduced words $g = b^{n_1}t^{\varepsilon_1}b^{n_2} \cdots b^{n_r}t^{\varepsilon_r}b^{n_{r+1}}$ that satisfy $n_1=0$ and $\mathfrak{h}(g)\leq M$. In particular, the set of words of height less than $M$ is precisely $\bigcup_{g \in F_0}\langle b \rangle g$; and by Remark~\ref{rem: dist normal form}, this is also the set of elements $w \in \Gamma$ that satisfy $w \cdot v \in B_{\mathcal{T}_{m,n}}(v, M)$.
\begin{lemma}\label{lem: never return}
    There exists some $f \in F_0$ such that \[\mathbb{P}\left(f G_k \cdot v \notin B_{\mathcal{T}_{m,n}}(v, M) \ \forall k > 0\right) > 0.\]
\end{lemma}

\begin{proof}
    Towards a contradiction, let us assume that, for every $f \in F_0$, one has \begin{equation}\label{eq: abs}\mathbb{P}\left(f G_k \cdot v \notin B_{\mathcal{T}_{m,n}}(v, M) \ \forall k > 0\right) = 0.\end{equation}
Let us define by induction the sequence of random variables $(\tau_i)_{i \in \mathbb{N}^*}$ as follows: \begin{align*}
    \tau_1 &= \inf \{k > 0 \mid G_k \cdot v \in B_{\mathcal{T}_{m,n}}(v, M) \}; \text{ and, for every $i \geq 1$} \\
    \tau_{i+1} &= \inf\{k > \tau_i \mid G_k \cdot v \in B_{\mathcal{T}_{m,n}}(v, M) \}.
\end{align*}
For every $i$, notice that $\tau_i$ is a stopping time with respect to the Markov chain $(G_k)_{k \in \mathbb{N}}$. Applying Equation~\ref{eq: abs} to $f=1$, we get that $\tau_1 < \infty$ almost surely. Assuming by induction that $\tau_i < \infty$ almost surely, we get that \begin{align*}
    \mathbb{P}\left(\tau_{i+1} < \infty\right) &= \mathbb{P}\left(\exists k > \tau_i, \ G_k \cdot v \in B_{\mathcal{T}_{m,n}}(v, M) \right) \\
    &= \sum_{f \in F_0}\mathbb{P}\left(\exists k > \tau_i,  S_1 \cdots S_k \cdot v \in B_{\mathcal{T}_{m,n}}(v, M) \text{ and } S_1 \cdots S_{\tau_i} \in \langle b \rangle f \right) \\
    &=\sum_{f \in F_0}\mathbb{P}\left(\exists k > \tau_i,  f \cdot S_{\tau_i +1} \cdots S_k \in B_{\mathcal{T}_{m,n}}(v, M) \text{ and } S_1 \cdots S_{\tau_i} \in \langle b \rangle f \right)
\\
&= \sum_{f \in F_0}\mathbb{P}\left(\exists k > \tau_i,  f \cdot S_{\tau_i +1} \cdots S_k \in B_{\mathcal{T}_{m,n}}(v, M)\right) \mathbb{P}\left(S_1 \cdots S_{\tau_i} \in \langle b \rangle f \right) \\&\text{ (by independence of the $S_k$'s)} \\
&= \sum_{f \in F_0}\mathbb{P}\left(S_1 \cdots S_{\tau_i} \in \langle b \rangle f \right) \\ &\text{ (as the $S_j$'s are identically distributed and using \eqref{eq: abs})} \\
&=1.
\end{align*}
    This proves that, almost surely, $G_k \cdot v \in B_{\mathcal{T}_{m,n}}(v, M)$ for infinitely many $k$'s which contradicts the fact that $(G_k \cdot v)_{k \in \mathbb{N}}$ converges almost surely to a random end.
\end{proof}

Before proving that the random walk $G_k \cdot v$ never returns to a subtree $\mathcal{T}^{(0)}$ as soon as it escapes a suitable neighborhood of $\mathcal{T}^{(0)}$ with non-zero (uniform) probability, we need the following deterministic result:

\begin{lemma}\label{lem: half tree}
        Let $\mathcal{T}^{(0)}$ be a subtree of $\mathcal{T}_{m,n}$.  Fix a sequence $(s_i)_{i \in \mathbb{N}^*} \in \Supp(\mu)^{\mathbb{N}}$ and some $w_0 \in \Gamma$ such that $w_0 \cdot v \notin \mathcal{T}^{(0)}$ and let us denote by $e$ the edge with source $w_0 \cdot v$ that points towards $\mathcal{T}^{(0)}$. Also assume that there exists some $l_0 > 0$ such that \begin{enumerate}
          \item $w_0 s_1 \cdots s_{l_0} \cdot v \in B_{\mathcal{T}_{m,n}}(w_0 \cdot v, M)$; and
          \item $w_0 s_1 \cdots s_k \cdot v \notin B_{\mathcal{T}_{m,n}}(w_0 \cdot v, M)$ for every $k > l_0$; and
          \item there exists some $m_0 > l_0$ such that $w_0 s_1 \cdots s_{m_0} \cdot v \in \mathcal{T}^{(0)}$.
          \end{enumerate} Let \[
\eta_e := \left\{
    \begin{array}{ll}
        t^{-1}bt & \mbox{if $e$ is positive} \\
        tbt^{-1} & \mbox{if $e$ is negative}
    \end{array}.
\right.
\]    
          Then, for every $k > l_0$, the vertex $w_0 \cdot \eta_e \cdot s_{1} \cdots s_k \cdot v$ does not belong to $\mathcal{T}^{(0)}$. 
\end{lemma}

\begin{proof}
     Let us denote by $\mathcal{T}_e$ the half-tree of $e$. As it contains $\mathcal{T}^{(0)}$, one has $w_0 s_1 \cdots s_{m_0} \cdot v \in \mathcal{T}_e$. Thus, as $w_0 s_1 \cdots s_k \cdot v$ never returns in $B_{\mathcal{T}_{m,n}}(w_0 \cdot v, M)$ for $k > l_0$, the fact that $d(w_0 s_1 \cdots s_k \cdot v, w_0 s_1 \cdots s_{k+1} \cdot v) \leq M$ for every $k$ (by definition of $M$) implies that the sequence of vertices $(w_0 s_1 \cdots s_k \cdot v)_{k > l_0}$ remains in $\mathcal{T}_e$. Let us assume that $e$ is positive. In particular, the geodesic between $w_0 \cdot v$ and $w_0 \cdot s_1 \cdots s_k \cdot v$ begins with the positive edge $e$ for every $k > l_0$. Thus, the geodesic between $v$ and $s_1 \cdots s_k \cdot v$ begins with the positive edge $w_0^{-1} \cdot e$ for every $k > l_0$. So by the ``more specifically'' part of Remark~\ref{rem: dist normal form}, the normal form of $s_{1} \cdots s_k$ begins with $b^rt$ for some $r \in \mathbb{Z}$ and for every $k > l_0$. Consequently, the normal form of $t^{-1}bts_{1} \cdots s_k$ begins with $t^{-1}$, which implies (using again Remark~\ref{rem: dist normal form}) that the geodesic that connects $v$ and $t^{-1}bts_{1} \cdots s_k \cdot v$ begins with a negative edge for every $k > l_0$. Thus, the geodesic that connects $w_0 \cdot v$ and $w_0 \cdot t^{-1}bts_{1} \cdots s_k \cdot v$ begins with a negative edge, hence $w_0 t^{-1}bts_{1} \cdots s_k \cdot v \notin \mathcal{T}_e$ so $w_0\eta_e s_{1} \cdots s_k \cdot v \notin \mathcal{T}^{(0)}$ for every $k > l_0$. If $e$ is negative, the same argument shows that the normal form of $tbt^{-1}s_{1} \cdots s_k$ begins with $t$ and consequently, that $w_0\eta_e s_{1} \cdots s_k \cdot v \notin \mathcal{T}^{(0)}$ for every $k > l_0$. 
\end{proof}

Now we provide a uniform lower bound on the probability of never returning to some subtree $\mathcal{T}^{(0)}$ knowing that the random walk has escaped a suitable neighborhood of $\mathcal{T}^{(0)}$.

\begin{lemma}\label{lem: borel cantelli}
    There exist $L > 0$ and $q > 0$ such that, for every $w_0 \in \Gamma$ that satisfies $w_0 \cdot v \notin \mathcal{N}_{\mathcal{T}_{m,n},L}\left(\mathcal{T}^{(0)}\right) $, one has \[\mathbb{P}\left(\ w_0 G_k \cdot v \notin \mathcal{T}^{(0)}  \ \forall k > 0 \right) \geq q.\]
\end{lemma}

\begin{proof}
     Using the fact that $\Supp(\mu)$ generates $\Gamma$ as a semi-group, let us fix a writing \[g = \mathfrak{g}_1^g \cdots \mathfrak{g}_{N_g}^g,\] for every $g \in F_0 \cup \{tbt^{-1},t^{-1}bt\}$, where $\mathfrak{g}_1^g, \ldots, \mathfrak{g}_{N_g}^g \in \Supp(\mu)$. Let us also define \[q_0 = \min_{(\eta,g) \in \{tbt^{-1}, t^{-1}bt\} \times F_0}\prod_{i=1}^{N_{\eta}}\mu\left(\mathfrak{g}_i^{\eta}\right) \cdot \prod_{i=1}^{N_g}\mu\left(\mathfrak{g}_i^{g}\right).\] Let \[\kappa = \max_{(\eta, g) \in \{tbt^{-1}, t^{-1}bt\} \times F_0, (i,j) \in \llbracket 1, N_{\eta}\rrbracket \times \llbracket 1, N_g\rrbracket}  \left\{\mathfrak{h}\left(\prod_{k=1}^i \mathfrak{g}_k^{\eta}\right),\mathfrak{h}\left(\eta \cdot \prod_{k=1}^j \mathfrak{g}_k^g\right), \mathfrak{h}\left(\prod_{k=1}^j \mathfrak{g}_k^g\right)\right\} \] and let us take \[L = \kappa + 1.\]

Now let $w_0 \in \Gamma$ such that $w_0 \cdot v \notin\mathcal{N}_{\mathcal{T}_{m,n},L}\left(\mathcal{T}^{(0)}\right)$. Let us denote by $e$ the edge with source $w_0 \cdot v$ that points towards $\mathcal{T}^{(0)}$.

Using Lemma~\ref{lem: never return}, let us fix some $f \in F_0$ such that \[\mathbb{P}\left(f  G_k \cdot v \notin B_{\mathcal{T}_{m,n}}(v, M) \ \forall k > 0  \right) > 0.\] We denote by $p$ this probability. Notice that, as $\Gamma$ acts on $\mathcal{T}_{m,n}$ by isometries, one also has \begin{align}\mathbb{P}\left(w_0f  G_k \cdot v \notin B_{\mathcal{T}_{m,n}}(w_0 \cdot v, M) \ \forall k > 0  \right) \nonumber &= \mathbb{P}\left(f  G_k \cdot v \notin B_{\mathcal{T}_{m,n}}(v, M) \ \forall k > 0  \right) \\ &=p. \label{eq: act isom} \end{align}

The definition of $\kappa$ implies that, for every $(\eta,g) \in \{tbt^{-1},t^{-1}bt\} \times F_0$: \begin{itemize}
    \item  for every $i \in \llbracket 1, N_g\rrbracket$, the vertex $u_i :=  w_0 \mathfrak{g}_1^g \cdots \mathfrak{g}_i^g \cdot v$ satisfies $d(w_0 \cdot v, u_i) \leq \kappa$;
    \item  for every $i \in \llbracket 1, N_g\rrbracket$, the vertex $v_i :=w_0 \eta \mathfrak{g}_1^g \cdots \mathfrak{g}_i^g \cdot v$ satisfies $d(w_0 \cdot v, v_i) \leq \kappa$;
    \item for every $i \in \llbracket 1, N_{\eta}\rrbracket$, the vertex $w_i :=w_0 \mathfrak{g}_1^{\eta} \cdots \mathfrak{g}_i^{\eta} \cdot v$ satisfies $d(w_0 \cdot v, w_i) \leq \kappa$.
\end{itemize}
 In particular, the triangle inequality implies that, if $x$ is one of the $u_i$'s, $v_i$'s, $w_i$'s, then \begin{align*}d\left(\mathcal{T}^{(0)},x\right) &\geq d(\mathcal{T}^{(0)},w_0 \cdot v) - \kappa \\ &\geq L-\kappa \\&=1, \end{align*} thus $x \notin \mathcal{T}^{(0)}$. Together with Lemma~\ref{lem: half tree}, this implies that
 \begin{align*}
     &\mathbb{P}\left(w_0 G_k \notin \mathcal{T}^{(0)}, \forall k > 0\right) \\
     &\geq \mathbb{P}\bigg(S_1 = \mathfrak{g}_1^{\eta_e}, \ldots, S_{N_{\eta_e}} = \mathfrak{g}_{N_{\eta_e}}^{\eta_e}, S_{N_{\eta_e}+1} = \mathfrak{g}_1^f, \ldots, S_{N_{\eta_e} + N_f} = \mathfrak{g}_{N_f}^f\\
     & \text{ and } w_0fS_{N_{\eta_e}+N_f+1}\cdots S_k \cdot v \notin B_{\mathcal{T}_{m,n}}(w_0 \cdot v, M), \forall k > N_{\eta_e}+N_f \\
     & \text{ and } \exists k > N_{\eta_e}+N_f, w_0fS_{N_{\eta_e}+N_f+1}\cdots S_k \cdot v \in \mathcal{T}^{(0)}\bigg) \\
     &\geq q_0 \mathbb{P}\left(w_0f G_k \cdot v \notin B_{\mathcal{T}_{m,n}}(w_0 \cdot v, M), \forall k > 0 \text{ and } \exists k > 0, w_0f G_k \cdot v \in\mathcal{T}^{(0)}\right),
 \end{align*}
 the last inequality resulting from the fact that the $S_i$'s are independent and identically distributed. Likewise, using again the previous remark: \begin{align*}
     &\mathbb{P}\left(w_0 G_k \notin \mathcal{T}^{(0)}, \forall k > 0\right) \\
     &\geq \mathbb{P}\left(S_1 = \mathfrak{g}_1^f, \ldots, S_{N_f} = \mathfrak{g}_{N_f}^f \text{ and } w_0fS_{N_f+1} \cdots S_k \notin \mathcal{T}^{(0)} \forall k > N_f\right) \\
     &\geq q_0\mathbb{P}\left(w_0fG_k \notin \mathcal{T}^{(0)} \forall k > 0\right).
 \end{align*}

 Decomposing \begin{align*}p &= \mathbb{P}\left(w_0f G_k \cdot v \notin B_{\mathcal{T}_{m,n}}(w_0 \cdot v, M) \forall k > 0 \text{ and } \exists k>0, w_0f G_k \cdot v \in \mathcal{T}^{(0)}\right)\\
 &+\mathbb{P}\left(w_0f G_k \cdot v \notin B_{\mathcal{T}_{m,n}}(w_0 \cdot v, M) \forall k > 0 \text{ and } w_0f G_k \cdot v \notin \mathcal{T}^{(0)}, \forall k > 0\right)
 \end{align*}
  we thus get \begin{align*}
      p &\leq \frac{2}{q_0} \mathbb{P}\left(w_0 G_k \cdot v \notin \mathcal{T}^{(0)}, \forall k > 0\right),
  \end{align*}
\textit{i.e.} \[\mathbb{P}\left(w_0 G_k \cdot v \notin \mathcal{T}^{(0)}, \forall k > 0\right) \geq \frac{pq_0}{2}.\]
Hence, setting $q := \frac{pq_0}{2}$ concludes the proof of the lemma.
\end{proof}

Now we are able to prove Proposition~\ref{convloin}. 

\begin{proof}[Proof of Proposition~\ref{convloin}]
  The fact that the random walk $G_k \cdot v$ converges almost surely to a random end $\xi \in \partial \mathcal{T}_{m,n}$ results from Theorem \ref{conv}. So it remains to show that $\xi \notin \partial \mathcal{T}^{(0)}$ almost surely.

      Let $L$ and $q$ be provided by Lemma~\ref{lem: borel cantelli}. As $\mathcal{N}_{\mathcal{T}_{m,n},L}\left(\mathcal{T}^{(0)}\right)$ is still a proper $\Lambda$-invariant subtree of $\mathcal{T}_{m,n}$ on which $\Lambda$ acts cocompactly, Lemma~\ref{sort}  tells us that: 
       \begin{equation}\label{eq: exit tree}\mathbb{P}\left(\exists k \geq 0, wG_k \cdot v \notin \mathcal{N}_{\mathcal{T}_{m,n}, L}\left(\mathcal{T}^{(0)}\right)  \right) = 1 \text{ for every $w \in \Gamma$.} \end{equation}
   
    Let us define inductively the random variables $\tau_i, \sigma_i$ which are the (possibly infinite) indices of the $i$-th exit from $\mathcal{N}_{\mathcal{T}_{m,n},L}\left(\mathcal{T}^{(0)}\right)$ (\textit{resp.} return to $\mathcal{T}^{(0)}$): \begin{align*}
        \tau_1 &= \min \left\{k > 0, G_k \cdot v \notin \mathcal{N}_{\mathcal{T}_{m,n},L}\left(\mathcal{T}^{(0)}\right)\right\}; \\
        \sigma_i &= \min\left\{k > \tau_i, G_k \cdot v \in \mathcal{T}^{(0)}\right\};\\
        \tau_{i+1} &=  \min\left\{k > \sigma_i, G_k \cdot v \notin \mathcal{N}_{\mathcal{T}_{m,n},L}\left(\mathcal{T}^{(0)}\right)\right\}.
    \end{align*} Let us define the event $\Sigma_i := \left\{\sigma_i < \infty\right\}$. Equation~\eqref{eq: exit tree} implies that
    \[\left\{G_k \cdot v \in \mathcal{T}^{(0)} \text{ for infinitely many $k$'s} \right\} = \bigcap_{i \in \mathbb{N}^*}\Sigma_i, \]
    thus we want to prove that $\mathbb{P}\left(\bigcap_{i \in \mathbb{N}^*}\Sigma_i\right)=0$. As $\Sigma_{i+1} \subseteq \Sigma_i$ for every $i \in \mathbb{N}^*$, one has $\mathbb{P}\left(\bigcap_{i \in \mathbb{N}^*} \Sigma_i\right) = \lim_{i \to +\infty}\mathbb{P}\left(\Sigma_i\right)$. For every $i \in \mathbb{N}$:
    \begin{align*}
        \mathbb{P}\left(\Sigma_{i+1}\right) &= \mathbb{P}\left(\tau_{i+1} < \infty \text{ and } \sigma_{i+1} < \infty\right)\\
        &= \sum_{w_0 \in \Gamma, w_0 \cdot v \notin \mathcal{N}_{\mathcal{T}_{m,n}, L}(\mathcal{T}^{(0)})}\mathbb{P}\big(\tau_{i+1} < \infty \text{ and } S_1 \cdots S_{\tau_{i+1}} = w_0 \\ & \ \ \ \ \ \ \ \ \ \ \ \ \ \ \ \ \ \ \ \ \ \ \ \ \ \ \ \ \ \ \ \ \ \ \ \  \ \ \text{ and } \exists \sigma > \tau_{i+1}, S_1 \cdots S_{\sigma} \cdot v \in \mathcal{T}^{(0)} \big)\\
        &= \sum_{w_0 \in \Gamma, w_0 \cdot v \notin \mathcal{N}_{\mathcal{T}_{m,n}, L}(\mathcal{T}^{(0)})}\mathbb{P}\big(\tau_{i+1} < \infty \text{ and } S_1 \cdots S_{\tau_{i+1}} = w_0 \\ &\ \ \ \ \ \ \ \ \ \ \ \ \ \ \ \ \ \ \ \ \ \ \ \ \ \ \ \ \ \ \ \ \ \ \ \  \ \  \text{ and } \exists \sigma > \tau_{i+1}, w_0 S_{\tau_{i+1}+1} \cdots S_{\sigma} \cdot v \in \mathcal{T}^{(0)} \big)\\
        &= \sum_{w_0 \in \Gamma, w_0 \cdot v \notin \mathcal{N}_{\mathcal{T}_{m,n}, L}(\mathcal{T}^{(0)})}\mathbb{P}\left(\tau_{i+1} < \infty \text{ and } S_1 \cdots S_{\tau_{i+1}} = w_0\right) \\ &\ \ \ \ \ \ \ \ \ \ \ \ \ \ \ \ \ \ \ \ \ \ \ \ \ \ \ \ \ \ \ \ \ \ \ \  \ \ \cdot \mathbb{P}\left( \exists \sigma > \tau_{i+1}, w_0 S_{\tau_{i+1}+1} \cdots S_{\sigma} \cdot v \in \mathcal{T}^{(0)} \right),
    \end{align*}
    the last equality resulting from the independence of the $S_k$'s and the fact that $\tau_{i+1}$ is a stopping time with respect to the random walk with step distribution $(S_k)_{k \in \mathbb{N}}$. As the $S_k$'s are identically distributed, we thus get:
    \begin{align*}
        \mathbb{P}\left(\Sigma_{i+1}\right) &= \sum_{w_0 \in \Gamma, w_0 \cdot v \notin \mathcal{N}_{\mathcal{T}_{m,n}, L}(\mathcal{T}^{(0)})}\mathbb{P}\left(\tau_{i+1} < \infty \text{ and } S_1 \cdots S_{\tau_{i+1}} = w_0\right)\\ &  \ \ \ \ \ \ \ \ \ \ \ \ \ \ \ \ \ \ \ \ \ \ \ \ \ \ \ \ \ \ \ \ \ \ \ \ \ \ \cdot \mathbb{P}\left( \exists \sigma' > 0, w_0 S_{1} \cdots S_{\sigma'} \cdot v \in \mathcal{T}^{(0)} \right) \\
        &\leq (1-q)\sum_{w_0 \in \Gamma, w_0 \cdot v \notin \mathcal{N}_{\mathcal{T}_{m,n}, L}(\mathcal{T}^{(0)})}\mathbb{P}\left(\tau_{i+1} < \infty \text{ and } S_1 \cdots S_{\tau_{i+1}} = w_0\right)\\ &\text{ (by Lemma~\ref{lem: borel cantelli}) } \\
        &= (1-q)\mathbb{P}\left(\tau_{i+1} < \infty\right)\\
        &\leq (1-q)\mathbb{P}(\Sigma_i).
    \end{align*}
     A straightforward induction thus implies that $\mathbb{P}\left(\Sigma_i\right) \leq (1-q)^{i-1}$ for every $i \in \mathbb{N}^*$, which proves that $\mathbb{P}\left(\bigcap_{i \in \mathbb{N}^*}\Sigma_i\right) = 0$.
    \end{proof}
    
Now we come back to preactions. The previous result has the following translation in the setting of Theorem~\ref{mixing}:
	
	\begin{corollary}\label{convloinquotient}
Let $\alpha$ be a transitive and non-saturated preaction on a pointed countable set $(X,x)$ whose $(m,n)$-graph is finite and let $\beta$ be its maximal forest saturation action (defined on a pointed countable set $(X',x)$ that contains $X$). Let us assume that \begin{itemize}
			\item either $|m|=|n|$; or
			\item $\alpha$ has infinite phenotype.
		\end{itemize}
		Then, the sequence $p_{\beta}(x \cdot G_k)$ converges almost surely to a random end of $\mathcal{G}_{\beta}$.  
	\end{corollary}
	
	\begin{proof}
  Let $\Lambda = \Stab_{\alpha}(x)$ and let $\pi : \mathcal{T}_{m,n} \to \Lambda \backslash \mathcal{T}_{m,n} = \mathcal{G}_{\beta}$ be the projection. Let $\mathcal{T}^{(0)} = \pi^{-1}(\mathcal{G}_{\alpha})$. By Remark~\ref{properinvsubtree}, $\mathcal{T}^{(0)}$ is a proper $\Lambda$-invariant subtree of $\mathcal{T}_{m,n}$ on which $\Lambda$ acts cocompactly (the quotient graph $\Lambda \backslash \mathcal{T}^{(0)}$ being $\mathcal{G}_{\alpha}$).
		So by Proposition \ref{convloin}, the sequence $\left(G_k \cdot v\right)_{k \in \mathbb{N}}$ converges almost surely to a random end $\xi \in \partial \mathcal{T}_{m,n} \smallsetminus \partial \mathcal{T}^{(0)}$. Let $e \in \mathcal{E}(\mathcal{T}_{m,n})$ such that $\src(e) \in \mathcal{T}^{(0)}$, $\trg(e) \notin \mathcal{T}^{(0)}$ and $\xi$ belongs to the half-tree $\mathcal{T}_e$ of $e$. For $k$ large enough, the sequence $G_k \cdot v$ remains in $\mathcal{T}_e$. Furthermore, $\pi$ induces a graph isomorphism between $\mathcal{T}_e$ and $\pi(\mathcal{T}_e)$ by Proposition~\ref{homeo} applied to $f = \pi(e)$. Thus, the sequence
        \begin{align*}
            p_{\beta}(x \cdot G_k) &= (\Lambda G_k)\langle b \rangle \\
            &= \Lambda (G_k \langle b \rangle) \\
            &=\pi(G_k \cdot v)
        \end{align*}
        converges almost surely to the random end $\pi\left(\xi\right)$ in $\pi\left(\mathcal{T}_e\right)$.
	\end{proof}

\subsubsection{Pasting preactions}\label{subsec: paste}
    
	The proof of Theorem \ref{mixing} now relies on the following key deterministic result:

    \begin{lemma}\label{merge}
        Let $\alpha_1$ and $\alpha_2$ be transitive non-saturated preactions on pointed countable sets $(X_i, x_i)$ ($i \in \{1,2\}$) whose $(m,n)$-graphs share the same phenotype $P$. Assume that this phenotype is infinite if $|m| \neq |n|$. For $i \in \{1,2\}$, let $\beta_i$ be the maximal forest saturation action of $\alpha_i$ (defined on a pointed coutable set $(Y_i,x_i)$ that contains $X_i$).

Let us consider reduced words $g_1,g_2,g_3$ such that: \begin{enumerate}
			\item for every subword $w$ of $g_2g_3$, one has $p_{\beta_1}(x_1 \cdot g_1w) \notin \mathcal{G}_{\alpha_1}$;
			\item for every subword $w$ of $g_2^{-1}g_1^{-1}$, one has $p_{\beta_2}(x_2 \cdot g_3^{-1}w) \notin \mathcal{G}_{\alpha_2}$;
			\item one has \begin{align*}d_{\mathcal{G}_{\beta_1}}(p_{\beta_1}(x_1 \cdot g_1), p_{\beta_1}(x_1 \cdot g_1g_2)) &\geq d_{\mathcal{G}_{\beta_1}}\left(\mathcal{G}_{\alpha_1}, p_{\beta_1}(x_1 \cdot g_1)\right)\\ &+d_{\mathcal{G}_{\beta_2}}\left(\mathcal{G}_{\alpha_2}, p_{\beta_2}\left(x_2 \cdot g_3^{-1}\right)\right)+2.\end{align*} 
		\end{enumerate}
		
		Then, there exists an action $\alpha$, whose $(m,n)$-graph is infinite, that extends both $\alpha_1$ and $\alpha_2$, and such that $x_1 \cdot g_1g_2g_3 = x_2$.   
    \end{lemma}
	
	\begin{proof}
		
		Let us define the subset \[X_1' = X_1 \cup \{p_{\beta_1}^{-1}(p_{\beta_1}\left(x_1 \cdot w)\right) \mid \text{$w$ subword of $g_1$}\}\] of $Y_1$ and the subset \[X_2' = X_2 \cup \{p_{\beta_2}^{-1}\left(p_{\beta_2}(x_2 \cdot w)\right) \mid \text{$w$ subword of $g_3^{-1}$}\}\] of $Y_2$. For $i \in \{1,2\}$, let us denote by $\beta_i'$ the restriction of $\beta_i$ defined on $X_i'$.
		
		As $P = \infty$ or $|m|=|n|$, the following arguments will take place in a half-tree of $\mathcal{G}_{\alpha_i}^c$ (for $i \in \{1,2\}$), which is isomorphic to a half-tree of $\mathcal{T}_{m,n}$ by Proposition \ref{homeo}; in particular, an edge path of $\mathcal{G}_{\alpha_i}^c$ that derives from a vertex of $\mathcal{G}_{\alpha_i}^c$ and a reduced word $\gamma$ is reduced, and its length is the height $\mathfrak{h}(\gamma)$ of $\gamma$.
        
        Especially, the edge path deriving from $(g_2,x_1 \cdot g_1)$ is reduced in $\mathcal{G}_{\alpha_1}^c$. Thus, one can write the normal form of $g_2$ as $g_2 = uv$, where all the subwords of $u$ are defined on $x_1 \cdot g_1$ and no nonempty subword of $v$ is defined on $x_1 \cdot g_1u$ (for the preaction $\beta_1'$). One has:
        \begin{align*}
			\mathfrak{h}(u)&= d_{\mathcal{G}_{\beta_1}}(p_1(x_1 \cdot g_1),p_1(x_1 \cdot g_1u)) \\ &\leq d_{\mathcal{G}_{\beta_1}}(p_1(x_1 \cdot g_1),\mathcal{G}_{\alpha_1}).
		\end{align*}
		Likewise, one can write the reduced normal form of $g_2$ as $g_2 = u'v'$, where all subwords of $v'^{-1}$ are defined on $x_2 \cdot g_3^{-1}$ and no nonempty subword of $u'^{-1}$ is defined on $x_2 \cdot g_3^{-1} v'^{-1}$ (for the preaction $\beta_2')$. One has
        \begin{align*}
			\mathfrak{h}(v')&= d_{\mathcal{G}_{\beta_2}}\left(p_{\beta_2}\left(x_2 \cdot g_3^{-1}\right),p_{\beta_2}\left(x_2 \cdot g_3^{-1}v'^{-1}\right)\right)  \\ &\leq d_{\mathcal{G}_{\beta_2}}\left(p_{\beta_2}\left(x_2 \cdot g_3^{-1}\right), \mathcal{G}_{\alpha_2}\right). 
		\end{align*}
		Thus, the third assumption implies that \begin{align*}\mathfrak{h}(g_2)&=d_{\mathcal{G}_{\beta_1}}(p_{\beta_1}(x_1 \cdot g_1),p_{\beta_1}(x_1 \cdot g_1g_2))\\& \geq \mathfrak{h}(u)+\mathfrak{h}(v')+2,\end{align*} so the initial subword $u$ of $g_2$ is in fact an initial subword of $u'$, and one can write the normal form of $g_2$ as $g_2 = u u'' v'$, where $\mathfrak{h}(u'') \geq 2$.

        Let us write $u'' = t^{\varepsilon}\mathfrak{m}t^{\eta}$ for some reduced word $\mathfrak{m}$, where \begin{itemize}
			\item $\varepsilon, \eta \in \{1,-1\}$;
			\item $t^{\varepsilon}$ is not defined on $x_1 \cdot g_1u$ (for the preaction $\beta_1'$);
			\item $t^{-\eta}$ is not defined on $x_2 \cdot g_3^{-1}v'^{-1}$ (for the preaction $\beta_2'$).
		\end{itemize} 
		
		Let us denote by $b^{n_1}t^{\varepsilon_1}\cdots b^{n_r}t^{\varepsilon_r}b^{n_{r+1}}$ the reduced form of $\mathfrak{m}$. By induction on $r$, we build a preaction $\gamma$ defined on a countable set $S$ \begin{itemize}
			\item whose $(m,n)$-graph is an edge path $e_1,\ldots ,e_{r}$, such that the orientation of the edge $e_i$ is the sign of $\varepsilon_i$;
			\item all of whose $\langle b \rangle$-orbits share the same cardinal $C$, which is the common label of the vertices of $\mathcal{G}_{\alpha_i}^c$ for $i\in \{1,2\}$ by Remark \ref{labelunimod} (and which is infinite if $P$ is);
			\item such that there exist $y_1,y_2 \in S$ such that \begin{itemize} 
				\item $t^{-\varepsilon}$ is not defined on $y_1$;
				\item $t^{\eta}$ is not defined on $y_2$;
				\item $y_2 = y_1 \cdot \mathfrak{m}$.
			\end{itemize}
		\end{itemize}
		Finally, we merge the preactions $\beta_1'$, $\gamma$ and $\beta_2'$ into a single preaction defined on $X_1' \sqcup S \sqcup X_2'$ by defining \begin{itemize}
			\item $x_1 \cdot g_1u \cdot t^{\varepsilon} = y_1$;
			\item $y_2 \cdot t^{\eta} = x_2 \cdot g_3^{-1}v'^{-1}$.
		\end{itemize}
        An illustration of this construction on the level of $(m,n)$-graphs is provided in Figure \ref{illustration}.
         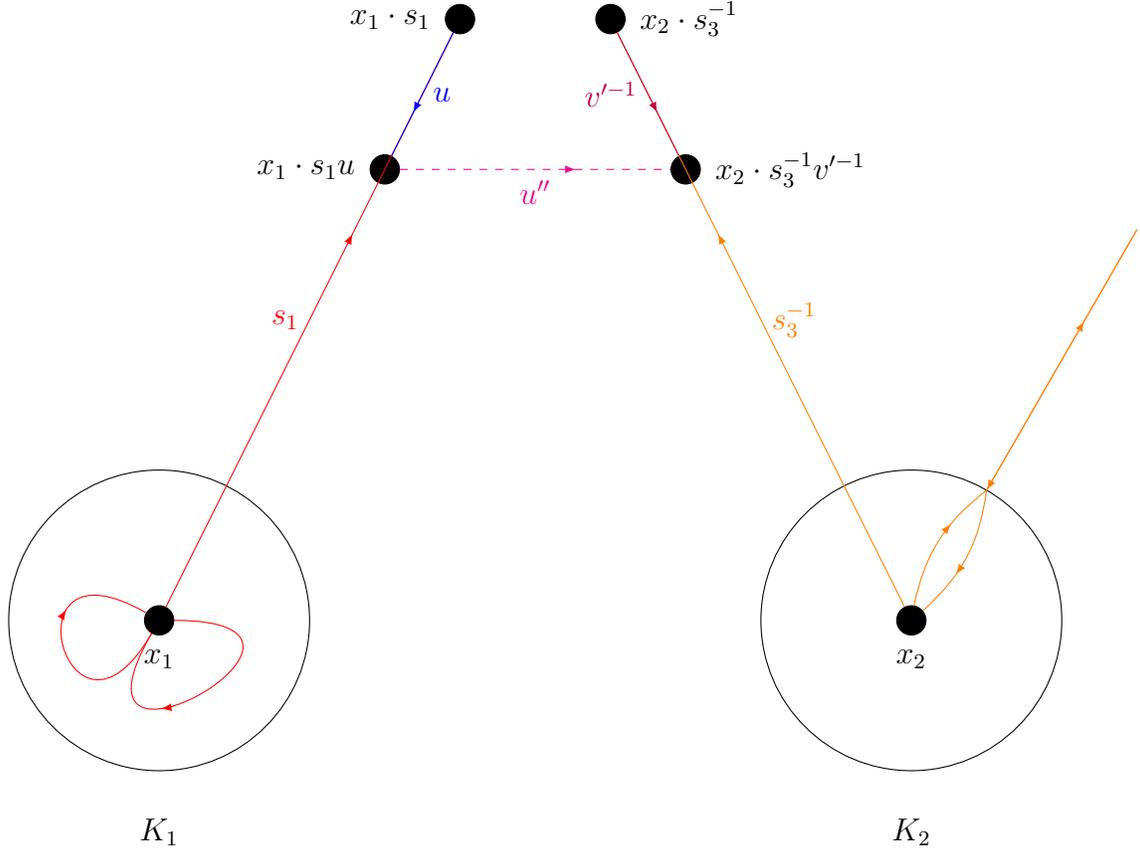
\begin{figure}[ht]
        \begin{tikzpicture}
\draw (0,0) circle (2) ;
\draw (8,0) circle (2) ;
\node[draw, circle, fill] (x1) at (0,0) {} ;
\draw (0,-0.25) node[below]{$x_1$} ;
\node[draw, circle, fill] (x2) at (8,0) {} ;
\draw (8,-0.25) node[below]{$x_2$} ;
\node[draw, circle, fill] (x1s1) at (3,7) {} ;
\draw (2.75,7) node[left]{$x_1 \cdot g_1$} ;
\node[draw, circle, fill] (x1s1u) at (2,4.67) {} ;
\draw (1.75,4.67) node[left]{$x_1 \cdot g_1u$} ;
\node[draw, circle, fill] (x2s3-1) at (5,7){} ;
\draw (5.25,7) node[right]{$x_2 \cdot g_3^{-1}$}{} ; 
\node[draw, circle, fill] (x2s3-1v'-1) at (6,4.67){} ;
\draw (6.25,4.67) node[right]{$x_2 \cdot g_3^{-1}v'^{-1}$} ;
\draw (0,-2.5) node[below]{$\mathcal{G}_{\alpha_1}$} ;
\draw (8,-2.5) node[below]{$\mathcal{G}_{\alpha_2}$} ;
\draw[>=latex, red, directed] (x1) -- (x1s1) node[midway, left]{$g_1$} ;
\draw[>=latex, blue, directed] (x1s1) -- (x1s1u) node[midway, right]{$u$} ;
\draw[>=latex, orange, directed] (x2) -- (x2s3-1) node[midway, right]{$g_3^{-1}$} ;
\draw[>=latex, purple, directed] (x2s3-1) -- (x2s3-1v'-1) node[midway, left]{$v'^{-1}$} ;
\draw[>=latex, magenta, dashed, directed] (x1s1u) -- (x2s3-1v'-1) node[midway, below]{$u''$} ;
\draw[>=latex, red, directed] (x1) to[out=240, in=150, looseness=20] (x1) ;
\draw[>=latex, red, directed] (x1) to[out=0, in=240, looseness=20] (x1) ;
\draw[>=latex, orange, directed] (x2) to[bend left=20] (9, 1.73) ;
\draw[>=latex, orange, directed] (9,1.73) -- (10,3.46) ;
\draw[->,>=latex, orange] (10,3.46) -- (9,1.73) ;
\draw[>=latex, orange, directed] (9,1.73) to[bend left=20] (x2) ;

\end{tikzpicture}
\caption{An illustration of the proof of Lemma \ref{merge}}
\label{illustration}
\end{figure}
        
        For this new preaction $\delta$ we get \begin{align*}
			x_2 \cdot g_3^{-1}v'^{-1} &= y_2 \cdot t^{\eta} \\
			&= y_1 \cdot \mathfrak{m} \cdot t^{\eta} \\
			&= x_1 \cdot g_1u t^{\varepsilon}\mathfrak{m} \cdot t^{\eta} \\
			&= x_1 \cdot g_1u u''
		\end{align*}
		which implies that
		\begin{align*}
			x_2 &= x_1 \cdot g_1u u''v'g_3 \\
			&= x_1 \cdot g_1g_2g_3
		\end{align*}
		and $\delta$ extends both $\alpha_1'$ and $\alpha_2'$. The $(m,n)$-graph of $\delta$ consists of \begin{itemize}
			\item the $(m,n)$-graph of $\beta_1'$;
			\item the $(m,n)$-graph of $\beta_2'$;
			\item the edge path $e_1,\ldots ,e_r$ (with $\src(e_1) \in \mathcal{G}_{\beta_1'}$ and $\trg(e_N) \in \mathcal{G}_{\beta_2'}$), all of whose vertices are labeled $C$.
		\end{itemize}
		Thus, as $|m|, |n|, r \geq 2$, the vertex $\trg(e_1)$ is not saturated. Hence, the maximal forest saturation action $\alpha$ of $\delta$ given by Lemma \ref{maxforest} has an infinite $(m,n)$-graph, thus satisfies the required conditions.
\end{proof}
	We are now ready to prove Theorem \ref{mixing}:
	
	\begin{proof}[Proof of Theorem \ref{mixing}]
		Let $\Lambda_1$ and $\Lambda_2$ be two subgroups of $\Gamma$ whose $(m,n)$-graphs are infinite. Let us assume that \begin{itemize}
			\item either $\bm{\Ph}_{m,n}(\Lambda_i) = \infty$ for $i \in \{1,2\}$; or
			\item $|m|=|n|$ and $\bm{\Ph}_{m,n}(\Lambda_1) = \bm{\Ph}_{m,n}(\Lambda_2)$,
		\end{itemize}
        and let us denote by $P$ the common phenotype of $\Lambda_1$ and $\Lambda_2$.
        
		For $i \in \{1,2\}$, let us denote by $\gamma_i$ the associated transitive right action on a pointed countable set $(X_i, x_i)$. Let us fix $R>0$ and let $\mathcal{F}_i$ be the $R$-ball of $\mathcal{G}_{\gamma_i}$ around $p_{\gamma_i}(x_i)$. We denote by
        \begin{align*}U_i^{(R)} = \big\{&\gamma \text{ transitive and saturated preaction on a pointed countable set } (W_{\gamma},w_{\gamma}) \mid \\& \Stab_{\gamma}(w_{\gamma}) \in \mathcal{K}(\Gamma) \text{ and } \left(p_{\gamma}^{-1}\left(B_{\mathcal{G}_{\gamma}}(p_{\gamma}(w_{\gamma}),R)\right), w_{\gamma}\right) \simeq \left(p_{\gamma_i}^{-1}(\mathcal{F}_i), x_i\right) \\& \text{ (as pointed labeled graphs)}\big\}.\end{align*} Identifying subgroups with pointed, transitive and saturated preactions, we recall that the sets $\left(U_i^{(R)}\right)_{R > 0}$ form a basis of neighborhoods of $\Lambda_i$ for the topology $\mathcal{T}_{sat}$ on the set of pointed transitive actions defined in Subsection \ref{projection} induced on $\mathcal{K}(\Gamma)$. Moreover, $U_i^{(R)}$ is included in $\bm{\Ph_{m,n}}^{-1}(P)$. 
        
        Let $\alpha_i$ be the subpreaction of $\gamma_i$ defined on a subset $X_i' \subseteq X_i$ containing $x_i$ and whose $(m,n)$-graph is $\mathcal{F}_i$. For $i \in \{1,2\}$, let $\beta_i$ be the maximal forest saturation action of $\alpha_i$ given by Lemma \ref{maxforest} (defined on a countable set $Y_i$ that contains $X_i'$). 
		
		Let us fix $\varepsilon > 0$. Using the fact that $\Supp(\mu)$ is bounded, let us denote by $M := \max_{\gamma \in \Supp(\mu)}\mathfrak{h}(\gamma)$. By Corollary \ref{convloinquotient} applied to $\beta_1$ and a sequence $(S_k)_{k \in \mathbb{N}}$ of independently $\mu$-distributed random variables on the one hand, and to $\beta_2$ and a sequence $\left(S'_k\right)_{k \in \mathbb{N}}$ of independently $\mu^{-1}$-distributed random variables (legit by Remark~\ref{rem: reversed}), denoting by $(G_k)_{k \in \mathbb{N}} = (S_1 \cdots S_k)_{k \in \mathbb{N}}$ and $(G_k')_{k \in \mathbb{N}} = \left(S_1' \cdots S_k'\right)_{k \in \mathbb{N}}$, we get that \begin{itemize}
		    \item there exists some $k_1 \in \mathbb{N}$ such that 
            \begin{equation}\label{k1m}\mathbb{P}\left(p_{\beta_1}(x_1 \cdot G_k) \notin \mathcal{N}_{\mathcal{G}_{\beta_1}, M}(\mathcal{F}_1), \ \forall k \geq k_1\right) \geq 1 - \varepsilon;\end{equation}

            \item there exists some $k_2 \in \mathbb{N}$ such that \begin{equation}\label{k2m}\mathbb{P}\left(p_{\beta_2}\left(x_2 \cdot G_k'\right) \notin \mathcal{N}_{\mathcal{G}_{\beta_2}, M}(\mathcal{F}_2), \ \forall k \geq k_2\right) \geq 1 - \varepsilon.\end{equation}
		\end{itemize}

         The choice of $M$ implies that 
\begin{align}\label{eq1}&\mathbb{P}\left(p_{\beta_1}(x_1 \cdot G_{k_1}U) \notin \mathcal{F}_{1} \text{ \ for every subword $U$ of $S_{k_1+1}\cdots S_k$ and every $k \geq k_1$}\right) \nonumber \\&\geq \mathbb{P}\left(p_{\beta_1}(x_1 \cdot G_k) \notin \mathcal{N}_{\mathcal{G}_{\beta_1}, M}(\mathcal{F}_1), \ \forall k \geq k_1\right) \nonumber \\ &\geq 1- \varepsilon \text{ by \eqref{k1m}}\end{align}
		and
		\begin{align}\label{eq2}&\mathbb{P}\left(p_{\beta_2}\left(x_2 \cdot G'_{k_2}U\right) \notin \mathcal{F}_{2} \text{ \ for every subword $U$ of $S_{k_2+1}'\cdots S_k'$ and every $k \geq k_2$}\right) \nonumber \\&\geq \mathbb{P}\left(p_{\beta_2}\left(x_2 \cdot G_k'\right) \notin \mathcal{N}_{\mathcal{G}_{\beta_2}, M}(\mathcal{F}_2), \ \forall k \geq k_2\right) \nonumber \\
        &\geq 1-\varepsilon \text{ by \eqref{k2m}.}\end{align}
 We draw the attention of the reader to the fact that $S_i, S_i'$'s may not be one of the standard generators $b,b^{-1},t,t^{-1}$, and that the term \textit{subword} has to be understood in the sense of Section \ref{normalform}, \textit{i.e.} with respect to the standard generators.

	Let us fix some $M'$ such that, for $i \in \{1,2\}$: \[M' > \max_{\left(h_1,\ldots ,h_{k_i}\right) \in \Supp(\mu)^{k_i}}\mathfrak{h}(h_1\cdots h_{k_i}).\]

    For $i \in \{1,2\}$, let us also define \[C_i = \max_{\left(h_1,\ldots, h_{k_i}\right) \in \Supp(\mu)^{k_i}}d_{\mathcal{G}_{\beta_i}}\left(p_{\beta_i}\left(x_i \cdot h_1\cdots h_{k_i}\right), \mathcal{F}_i\right).\] 
    By Corollary~\ref{convloinquotient}, there exists $k_0 > k_1+k_2$ such that, for any $k \geq k_0$: \begin{equation}\label{eq3}
        \mathbb{P}\left(d_{\mathcal{G}_{\beta_1}}\left(p_1(x_1), p_1\left(x_1 \cdot S_1\cdots S_{k-k_2}\right)\right) \geq C_1+C_2+2M'+2\right) \geq 1 - \varepsilon.
    \end{equation}

    Now we fix $k \geq k_0$ and, for every $i \in \llbracket 1, k \rrbracket$, we let $S_i' = S_{k-i+1}^{-1}$. Let us introduce the events 
        \[A := \left\{p_{\beta_1}(x_1 \cdot S_1\cdots S_{k_1}U) \notin \mathcal{F}_{1} \text{ \ for every subword $U$ of $S_{k_1+1}\cdots S_k$}\right\};\]
        \[B := \left\{p_{\beta_2}(x_2 \cdot S_1'\cdots S_{k_2}'U) \notin \mathcal{F}_{2} \text{ \ for every subword $U$ of $S_{k_2+1}'\cdots S_k'$}\right\};\]
        \[C := \left\{d_{\mathcal{G}_{\beta_1}}\left(p_1(x_1), p_1\left(x_1 \cdot S_1\cdots S_{k-k_2}\right)\right) \geq C_1+C_2+2M'+2\right\}.\]
    By Equations \eqref{eq1}, \eqref{eq2} and \eqref{eq3}, we get \begin{align*}\label{conditions}
			\mathbb{P}\left(A \cap B \cap C\right) \geq 1-3\varepsilon.
		\end{align*}
    Let $(s_i)_{i \in \mathbb{N}} \in \Supp(\mu)^{\mathbb{N}}$ satisfying these three conditions, \textit{i.e.} \begin{itemize}
			\item $p_{\beta_1}\left(x_1 \cdot s_1\cdots s_{k_1}u\right) \notin \mathcal{F}_{1}$ for every subword $u$ of $s_{k_1+1}\cdots s_k$;
			\item $p_{\beta_2}\left(x_2 \cdot s_k^{-1}\cdots s_{k-k_2+1}^{-1}v\right) \notin \mathcal{F}_{2}$ for every subword $v$ of $s_{k-k_2}^{-1}\cdots s_1^{-1}$;
			\item $d_{\mathcal{G}_{\beta_1}}\left(p_{\beta_1}(x_1), p_{\beta_1}\left(x_1 \cdot s_1\cdots s_{k-k_2}\right)\right) \geq C_1+C_2+2M'+2$.
		\end{itemize}
		Then: \begin{align*}
			&d_{\mathcal{G}_{\beta_1}}\left(p_{\beta_1}(x_1 \cdot s_1\cdots s_{k_1}), p_{\beta_1}(x_1 \cdot s_1\cdots s_{k_1} \cdot s_{k_1+1} \cdots s_{k-k_2})\right) \\&\geq d_{\mathcal{G}_{\beta_1}}\left(p_{\beta_1}(x_1),p_{\beta_1}(x_1 \cdot s_1\cdots s_k)\right) - d_{\mathcal{G}_{\beta_1}}\left(p_{\beta_1}(x_1), p_{\beta_1}(x_1 \cdot s_1\cdots s_{k_1})\right) \\& - d_{\mathcal{G}_{\beta_1}}\left(p_{\beta_1}(x_1\cdot s_1\cdots s_{k-k_2}), p_{\beta_1}(x_1 \cdot s_1\cdots s_k)\right) \text{\ (by the triangle inequality)}\\
			&\geq \left(C_1+C_2+2M'+2\right) - M' - M' \\
			&=  C_1+C_2+2 \\
            &\geq d_{\mathcal{G}_{\beta_1}}\left(\mathcal{F}_1,p_{\beta_1}(x_1 \cdot s_1 \cdots s_{k_1})\right) +d_{\mathcal{G}_{\beta_2}}\left(\mathcal{F}_2,p_{\beta_2}\left(x_2 \cdot s_{k}^{-1} \cdots s_{k-k_2+1}^{-1}\right)\right)+2,
		\end{align*}
		which implies that the preactions $\alpha_1, \alpha_2$ and the reduced forms $g_1,g_2,g_3$ of the three elements $s_1 \cdots s_{k_1}$, $s_{k_1+1} \cdots s_{k-k_2}$ and $s_{k-k_2+1} \cdots s_k$ of $\Gamma$ satisfy the assumptions of Lemma \ref{merge}. Thus, by Lemma \ref{merge}, there exists a saturated preaction $\alpha$ defined on a pointed countable set $(X,x_1)$ that contains $X_1'$ and $X_2'$ as disjoint subsets, whose $(m,n)$-graph is infinite, and such that \begin{itemize}
			\item $\alpha$ extends both $\alpha_1'$ and $\alpha_2'$;
			\item $x_1 \cdot s_1\cdots s_k = x_2$.
		\end{itemize}

        We proved that for every $R>0$ and every $\varepsilon > 0$, there exists $k_0 \in \mathbb{N}$ such that, for every $k \geq k_0$ one has \[\mathbb{P}\left(\exists \Lambda \in \mathcal{K}(\Gamma)\cap \bm{\Ph}_{m,n}^{-1}(P): \Lambda \in U_1^{(R)} \cap G_k^{-1}U_2^{(R)}G_k\right) \geq 1-3\varepsilon.\]
        Thus, as the sets $\left(U_i^{(R)}\right)_{R > 0}$ form a basis of neighborhoods of $\Lambda_i$ for $\mathcal{T}_{sat}$, which is finer than the Chabauty topology (\textit{cf.} Remark \ref{topfin}), the conjugation action is topologically $\mu$-mixing on $\mathcal{K}(\Gamma) \cap \bm{\Ph_{m,n}}^{-1}(P)$.\end{proof}

	\bibliographystyle{alpha}
\bibliography{ref}
	
	\bigskip
	{\footnotesize
		
		\noindent
		{\textsc{ENS-Lyon,
				Unité de Mathématiques Pures et Appliquées,  69007 Lyon, France}}
		\par\nopagebreak \texttt{sasha.bontemps@ens-lyon.fr}
	}

\end{document}